\documentclass{article}
\usepackage[utf8]{inputenc}
\usepackage{amsfonts}
\usepackage{amssymb}
\usepackage{amsmath}
\usepackage{amsthm}
\usepackage{commath}
\usepackage{xcolor}
\usepackage{bbold}
\usepackage{mathrsfs}
\usepackage{enumitem}
\usepackage[noabbrev]{cleveref}

\newcommand{\N}{\mathbb{N}}
\newcommand{\T}{\mathbb{T}}
\newcommand{\Z}{\mathbb{Z}}
\newcommand{\Q}{\mathbb{Q}}
\newcommand{\R}{\mathbb{R}}
\newcommand{\C}{\mathbb{C}}
\newcommand{\A}{\mathcal{A}}

\newcommand{\grad}{\textup{grad}}

\newcommand{\Hil}{\mathbb{H}}

\newcommand{\Id}{\mathbb{1}}

\newcommand{\wt}[1]{\widetilde{#1}}
\newcommand{\Lie}{\mathcal{L}}

\newcommand{\vphi}{\varphi}
\newcommand{\hdd}[1]{\mathcal{#1}}
\newcommand{\niek}[1]{\mathscr{#1}}
\newcommand{\ratio}{\mathcal{R}}

\newcommand{\eps}{\epsilon}

\newtheorem{theorem}{Theorem}
\newtheorem{definition}[theorem]{Definition}
\newtheorem{lemma}[theorem]{Lemma}

\newtheorem{remark}[theorem]{Remark}
\newtheorem{proposition}[theorem]{Proposition}
\newtheorem{example}[theorem]{Example}
\newtheorem{counterexample}[theorem]{Counterexample}

\numberwithin{theorem}{section}

\title{Regularized polysymplectic geometry and first steps towards Floer theory for covariant field theories}
\author{Ronen Brilleslijper and Oliver Fabert}
\date{\today}

\begin{document}

\maketitle

\begin{abstract}It is the goal of this paper to present the first steps for defining the analogue of Hamiltonian Floer theory for covariant field theory, treating time and space relativistically. While there already exist a number of competing geometric frameworks for covariant field theory generalizing symplectic geometry, none of them are readily suitable for variational techniques such as Hamiltonian Floer theory, since the corresponding action functionals are too degenerate. Instead, we show how a regularization procedure introduced by Bridges leads to a new geometric framework for which we can show that the finite energy $L^2$-gradient lines of the corresponding action functional, called Floer curves, converge asymptotically to space-time periodic solutions.  As a concrete example we prove the existence of Floer curves, and hence also of space-time periodic solutions, for a class of coupled particle-field systems defined in this new framework.
\end{abstract}
\tableofcontents

\section{Introduction}
The world of classical physics essentially contains two major areas: classical mechanics and classical field theory (CFT). The rigorous mathematical formulation of the former led to the development of Hamiltonian dynamics and symplectic geometry. As in other fields of geometry, techniques from variational calculus can provide important tools for proving theorems. In variational calculus, solutions to equations are found by translating them into a minimization problem of some functional. In the 1980's Andreas Floer discovered that $L^2$-gradient lines of the so-called action functional satisfy a perturbed version of the Cauchy-Riemann equations (\cite{floer1988morse}). Thus, to study these gradient lines he could make use of the previously discovered theory of Gromov on pseudo-holomorphic curves. As these gradient lines converge to critical points of the action functional, this allowed him to prove the Arnold conjecture about the number of periodic solutions to Hamilton's equations for a large class of symplectic manifolds (see for example \cite{floer1988morse, audindamian}). The resulting theory, which was named Floer theory after him, brought a rich family of tools to symplectic geometers. Having all of this theory for classical mechanics makes one wonder if similar tools can be developed for CFT. At present there is no analogue of Floer theory in the setting of CFT. There are however several natural approaches to take when attempting to formulate such a theory. We will introduce some of them here. 

First of all, CFT can be seen as an infinite-dimensional version of classical mechanics and its equations can be translated to symplectic geometry on infinite-dimensional manifolds. Thus, one natural approach of defining a Floer theory would be to develop it on infinite-dimensional symplectic spaces. This approach is taken by \cite{paperoliverniek} and provides a working theory of which we will review the basics in \cref{sub:niekapproach}. As explained at the end of that section, the major disadvantage of this theory is that it is not symmetric in time and space. As CFT is capable of describing relativistic phenomena such as electromagnetism, it would be desirable to formulate a theory that is capable of dealing with coordinate transformations of the space-time manifold considered. Moreover, a theory that uses the symmetry of the underlying space instead of breaking it, is also from a mathematical point of view more natural to develop. Even so, the theory developed in \cite{paperoliverniek} is a very good step in the direction of formulating a Floer theory for CFT and allows us to prove some existence results for our theory as well (see \cref{sub:existence}). 

When we think of CFT in a covariant way (i.e. not splitting time and space), there are several methods of formulating the geometry behind it. Different articles use different methods, such as multisymplectic geometry (\cite{marsden1998multisymplectic}), polysymplectic geometry (\cite{gunther1987polysymplectic,mcclain2021global}) and others (\cite{royksymplectic,KRUPKOVA200293}). The mathematics world has not (yet) decided on one universally standard choice of language, but in this article we chose to stick with the polysymplectic formulation. This is mostly due to the inspiration we got from \cite{gunther1987polysymplectic} and \cite{mcclain2021global}, who both use polysymplectic geometry. As can be seen in the aforementioned articles by G\"unther and McClain, they made good progress in developing a geometry for CFT. However, we believe that it is impossible to build a Floer theory for their equations. In fact, the polysymplectic Hamiltonian equations are not suited to be studied by any kind of variational techniques. The reason for this, which will be explained more in \cref{sub:degeneracy}, is that the equations are too degenerate: the differential operator that lies at the core of the equations has an infinite-dimensional kernel which obstructs much of the needed analysis. In particular, some of the fundamental building blocks of Floer theory (see \cref{lem:GW}) can not be reproduced for G\"unther's or McClain's equations, which provides major obstacles for the analysis of the Floer equation. Fortunately, this obstacle can be overcome by enhancing these equations to a richer structure that solves this degeneracy problem. In \cref{sec:Bridges}, we will explain how an idea from \cite{bridgesTEA} enables us to augment the polysymplectic form of G\"unther\footnote{In a follow-up article we will deal with the global case described by McClain's equation.} and produce a form that is in a sense more natural and even allows us to connect the covariant picture with the infinite-dimensional approach. The main purpose of this article is to convey the message that the polysymplectic form coming from what we call \emph{Bridges regularization} provides the correct formalism to describe CFT covariantly in a way that allows for variational calculus. Apart from the fact that this form is better suitable for calculus of variations (in particular Floer theory) and allows for a connection with the infinite-dimensional approach of \cite{paperoliverniek}, it also arises more naturally from a geometric point of view as is explained at the end of \cref{sub:underlyingstr}.

The idea behind Bridges regularization is as follows. In classical mechanics, the position of a particle can be interpreted as a 0-form and its momentum as a 1-form. Since the only variable is the time-variable, there are no higher-degree forms to consider. When we move to field equations, the field values and momenta can still be interpreted as 0-forms and 1-forms respectively. However, there are now both time and space variables in the picture, meaning that there exist higher-degree forms as well. In the existing approaches for the covariant formulation of CFT only the field values and momenta are taken into account, whereas the higher-degree forms are ignored. This is the source of the degeneracy problem in these frameworks. In Bridges regularization these higher-degree forms are taken into the equation, which resolves the degeneracy. In particular, the components of the resulting polysymplectic form are themselves symplectic forms as opposed to G\"unthers polysymplectic form, whose components are merely closed 2-forms. 

For this article, we deliberately chose to work only with scalar theories on linear space and attack the problems hands-on before digressing about the underlying geometry. The reason for this is that we want the reader to be able to reach the main messages of this article before getting lost into abstract constructions and computations. However, the basic background of polysymplectic geometry on linear spaces that is avoided in \cref{sec:CMtoFT,sec:Bridges} is treated in \cref{sec:polysympl}. Also, we are in fact able to generalise much of the theory in this article to field theories on general manifolds taking values in non-trivial vector bundles. To include these generalisations into this article would make it too long and distract from the main message. Therefore, we decided to publish the global geometric picture in a subsequent article. The article at hand is meant to motivate the use of a regularized polysymplectic geometry over the existing frameworks and to prove some first results in this framework. The connection that this framework allows between the covariant and the infinite-dimensional picture will also enable us to indeed prove the existence of Floer curves and periodic solutions to certain field equations.


This article is structured as follows. \Cref{sec:CMtoFT} starts with a small recap of classical mechanics and how Floer theory comes into the picture. \Cref{lem:GW} gives two fundamental lemmas that are crucial for the validity of Floer theory. They will be used in other sections to test possible approaches to Floer theory for CFT. Next, field theories are introduced, starting with the infinite-dimensional approach of \cite{paperoliverniek} (\cref{sub:niekapproach}) and followed by the polysymplectic formulation (\cref{sub:relFT}). The section ends with an illustration of the degeneracy problem of the latter (\cref{sub:degeneracy}). This leads to the introduction of Bridges regularization in \cref{sec:Bridges} and we show that indeed the degeneracy problem is solved (\cref{sub:degsolved}). \Cref{sec:Bridges} starts very hands-on by manipulating the equations, but the underlying structure of these alterations is explained later on in \cref{sub:underlyingstr}. We show that this newly obtained structure also fits into the infinite-dimensional theory of \cite{paperoliverniek} in \cref{sub:infvp} and still provides the same set of solutions for the field equation that we started with (\cref{sub:spaceofsol}). In \cref{sec:polysympl}, we explain more rigorously how polysymplectic geometry works and give some background for constructions that were used more ad-hoc in the previous sections. In particular, it is explained how the action functionals of \cref{sec:CMtoFT,sec:Bridges} come about and why they have the correct critical points. Finally, \cref{sec:particlefield} illustrates the theory by applying it to the example of coupled particle-field systems. This provides an interesting application, both from a physical and mathematical point of view. Also, \cref{sub:existence} shows that indeed our version of the Floer equation allows us to prove the existence of solutions to field equations, which is the problem that we started with. Furthermore, \cref{sub:Lorentz} illustrates the advantage of our developed theory over \cite{paperoliverniek}.


\section{From classical mechanics to field theory}\label{sec:CMtoFT}

We start with a small recap of classical mechanics and the basic ideas of Floer theory. The most fundamental equation in classical mechanics is Newton's second law:
\begin{align}\label{eq:Fma}
    \frac{d^2}{dt^2}q(t) = -V_t'(q(t)).
\end{align}
Here, $q(t)$ describes the position at time $t$ of an object moving in some potential field given by $V_t$. We start by briefly discussing how this equation can be formulated with the help of a Hamiltonian and symplectic form and what the role of Floer theory is for finding solutions.

First of all, to transform this equation into a first order ODE, we define the \textit{momentum} $p(t)=q'(t)$. Now, \cref{eq:Fma} is equivalent to the system of equations
\begin{align*}
    \frac{d}{dt}q(t) = p(t) && \frac{d}{dt}p(t) = -V_t'(q).
\end{align*}
This system of equations can be unified into one equation using the \textit{Hamiltonian} function $H_t(q,p)=\frac{1}{2}p^2+V_t(q)$. It becomes
\begin{align}\label{eq:symplJdt}
  J  \frac{d}{dt}u(t) = \nabla H_t(u(t)),
\end{align}
where $u(t)=(q(t),p(t))$ and $J=\begin{pmatrix}0&-1\\1&0\end{pmatrix}$ is the standard complex structure on $\R^2$. As treated in standard textbooks on symplectic geometry (for example \cite[Chapter 18]{da2008lectures}), the geometric picture behind this equation lies in the two-form $\omega = \langle\cdot,J\cdot\rangle=dp\wedge dq$, called the \textit{standard symplectic form} on $\R^2$. It enables us to write the equation above as
\begin{align}\label{eq:sympl}
    (dH_t)_{u(t)}=\omega(\cdot,\frac{d}{dt}u(t)).
\end{align}

In Floer theory, one tries to prove the existence of periodic solutions of the equation above. First, the essential observation is made that solutions $u:S^1\to\R^2$ to \cref{eq:sympl} are critical points of the action functional
\begin{align*}
    \A_{H}(u) &= \int_{S^1}(u^*\lambda)_t-\int_0^1H_t(u(t))\,dt\\
    &=\int_0^1 \left(p(t)q'(t)-H_t(u(t))\right)\,dt.
\end{align*}
Here, $\lambda=p\,dq$ is a primitive for $\omega$ and $H_t=H_{t+1}$ is assumed to be periodic in $t$. Just like in Morse theory, Floer theory proves the existence of these critical points by studying gradient lines. If $\mathcal{X}=-\grad\,\A_H$ is defined to be the negative $L^2$-gradient of $\A_H$, then its smooth trajectories $\wt{u}:\R\times S^1\to\R^2$, called \textit{Floer curves}, are characterized by
\begin{align}\label{eq:symplFloer}
    \partial_s \wt{u}(s,t)+J\partial_t\wt{u}(s,t)-\nabla H_t(\wt{u}(s,t))=0.
\end{align}
\Cref{eq:symplFloer} is called the \textit{Floer equation} and forms the foundation of Floer theory. It is interesting to note, that when both $\wt{u}$ and $H$ are independent of $t$, the equation reduces to the Morse theory of the function $H$. An essential observation made by Floer is that \cref{eq:symplFloer} reduces to the Cauchy-Riemann equation when $H\equiv 0$. Its analysis builds upon the well-established theory of Gromov and Witten.\footnote{See \cite{mcduff2012j} for a detailed exposition of this theory.} We will give two important ingredients for the theory below. For a rigorous exposition of Floer theory we refer to \cite{audindamian}. 

There are two results in symplectic geometry that are crucial for the validity of Floer theory. Before restating them, we must define the \textit{energy} of a Floer curve:
\begin{align*}
    E(\wt{u}) = -\int_{-\infty}^{+\infty}\frac{d}{ds}\A_H(\wt{u}(s,\cdot))\,ds = \int_{-\infty}^{+\infty}\left(\int_{S^1}|\partial_s\wt{u}(s,t)|^2\,dt\right)\,ds
\end{align*}
From now on we will denote $\wt{u}_s:=\wt{u}(s,\cdot)$.

\begin{lemma}\label{lem:GW}
Let $\wt{u}$ be a Floer curve with finite energy: $E(\wt{u})<+\infty$.
\begin{enumerate}[label=(\roman*), ref={\thetheorem.\roman*}]
    \item\label{GW:zero} If $H\equiv 0$ and $\wt{u}_s$ has mean zero for every $s$ then
    \[\lim_{s\to\infty}\wt{u}_s=0\]
    in $C^\infty(S^1,\R^2)$.
    \item If all critical points of $\A_H$ are non-degenerate\footnote{Meaning that the Hessian of $\A_H$ is invertible.}, then there is some $u\in C^\infty(S^1,\R^2)$ satisfying \cref{eq:sympl} such that
    \[\lim_{s\to\infty}\wt{u}_s=u\]
    in $C^\infty(S^1,\R^2)$.\label{GW:ndg}
\end{enumerate}
\end{lemma}
For both parts of the lemma we refer to standard textbooks on holomorphic curves and Floer theory like \cite{audindamian,mcduff2012j,salamon1999lectures}. Lemma \ref{GW:ndg} implies that if there are Floer curves with finite energy, then automatically the existence of periodic solutions to \cref{eq:sympl} is established. Lemma \ref{GW:zero} plays an important role in proving the validity of Floer theory. As mentioned briefly above, Floer theory builds on Gromov's theory of pseudoholomorphic curves and lemma \ref{GW:zero} is needed for establishing so-called "bubbling off" results. If lemma \ref{GW:zero} fails to hold, the construction of a Floer theory is hopeless, thus we will use equivalents of lemma \ref{GW:zero} to check the validity of various generalisations of Floer theory in the subsequent text.

\subsection{Non-relativistic field theory}\label{sub:niekapproach}
Our guiding example for field theories will be the free wave equation
\begin{align}\label{eq:wave}
    -\partial_t^2\vphi(t,x)+\partial_x^2\vphi(t,x) = 0,
\end{align}
where $\vphi(t,x)$ denotes a scalar field on 2-dimensional space-time. Of course, the equation becomes interesting only when we add a non-linearity to the right-hand side. In particular we will be interested in non-linearities that come from a coupling of the field with a particle (see \cref{eq:particlefield}). This provides interesting non-linearities in the field theory, but also influences the mechanical system of the particle, thereby including symplectic geometry into the story. However, for the purpose of this chapter, the free wave equation already suffices. Just like in the previous paragraphs, we want to describe \cref{eq:wave} using some geometric framework and examine the existence of periodic solutions using an analogue of Floer theory. For our first attempt in doing so, we follow the exact same strategy as above. 

We start by defining the momentum $\pi=\partial_t \vphi$ in order to get rid of the second-order derivative in $t$. \Cref{eq:wave} becomes equivalent to  
\begin{align*}
    \partial_t\vphi(t,x) = \pi(t,x) && \partial_t\pi(t,x)=\partial_x^2\vphi(t,x).
\end{align*}
Just as in the case of classical mechanics discussed above, this system of equations can be unified into one equation using a Hamiltonian function. However, the Hamiltonian has to be defined on an infinite-dimensional space now. Let $\Hil=L^2(S^1,\R)\oplus L^2(S^1,\R)$ be the space of square-integrable functions in the space-variable. We denote elements of $\Hil$ by $\niek{Z}$, where $\niek{Z}(x)=(\vphi(x),\pi(x))$.  Note that we have taken the domain of these functions to be the circle instead of $\R$, since eventually we are interested in periodic solutions anyways. Define $\niek{H}:\Hil\to\R$ by 
\[
\niek{H}(\niek{Z})=\int_{S^1}\left[\frac{1}{2}\pi^2(x)+\frac{1}{2}\left(\frac{d}{dx}\vphi(x)\right)^2\right]\,dx.
\]
Then
\[
\grad\,\niek{H}(\niek{Z})=\begin{pmatrix}-\frac{d^2}{dx^2}\vphi\\ \pi\end{pmatrix},
\]
so \cref{eq:wave} becomes equivalent to
\begin{align}\label{eq:niekJdt}
\niek{J}\frac{d}{dt}\niek{Z}(t) = \grad\,\niek{H}(\niek{Z}(t)),
\end{align}
where now we write $\niek{Z}(t)(x)=\niek{Z}(t,x)$ and where $\niek{J}(\vphi,\pi)=(-\pi,\vphi)$ is the standard complex structure on $\Hil$.

Define $\omega_\Hil=\langle\cdot,\niek{J}\cdot\rangle_{L^2}=d\pi\wedge d\vphi$ to be the standard symplectic form on $\Hil$. Then finally we can write \cref{eq:wave} as 
\begin{align}\label{eq:poly}
    (d\niek{H})_{\niek{Z}(t)}=\omega_{\Hil}(\cdot,\frac{d}{dt}\niek{Z}(t)).
\end{align}
One can compare \cref{eq:poly} to \cref{eq:sympl} to see that the two equations are very similar. Whereas \cref{eq:sympl} brings the equation into the world of symplectic geometry, \cref{eq:poly} translates the wave equation into the realm of infinite-dimensional symplectic geometry. In a similar way as before, an action functional can be defined and periodic solutions of the free wave equation can be found by studying gradient trajectories of this functional. This approach has been taken by Fabert and Lamoree in \cite{paperoliverniek} and leads to an interesting and working theory. There is one major problem with this theory though: it breaks the symmetry between time and space. In the free wave equation time and space play a very similar role. In fact, every linear coordinate transformation in $O(1,1)$ preserves the equation. On the other hand, \cref{eq:poly} is highly asymmetric in time and space and a choice of coordinates must be made before applying the theory. This is undesirable, considering that the wave equation describes fundamentally relativistic phenomena such as electromagnetism and its interaction with charged particles. \Cref{sec:particlefield} will treat an example of the theory in which one cannot assume a fixed splitting of time and space. This example does not fit into the framework developed by Fabert and Lamoree and therefore illustrates the need to find a covariant formulation of the theory. However, also from a purely theoretical point of view, it is desirable to find a treatment of covariant field theories that uses their symmetry rather than breaking it. The remainder of this article will deal with finding a theory that realizes this. 

\subsection{Relativistic field theory}\label{sub:relFT}
In the previous paragraph, the first step in reformulating the wave equation was to introduce a momentum variable for the derivative with respect to time. As a start of our symmetric approach, a logical step would be to introduce momentum variables for both the time and space derivatives; $\pi_1(t,x)=\partial_t\vphi(t,x)$ and $\pi_2(t,x)=\partial_x\vphi(t,x)$. The free wave \cref{eq:wave} is then equivalent to the system of equations
\begin{alignat}{3}
    -&\partial_t\pi_1(t,x) &+\partial_x\pi_2(t,x) &= 0\nonumber\\
    &\partial_t\vphi(t,x) & &=\pi_1(t,x)\label{eq:hddsystem}\\
    & &-\partial_x\vphi(t,x)&=-\pi_2(t,x).\nonumber
\end{alignat}
The reason for the minus signs in the last row will become apparent below. Notice that the set of solutions to this system of equations is still in one-to-one correspondence with the set of solutions of \cref{eq:wave}. Just as before, we want to unify these equations using a type of Hamiltonian function. To this extent, define $\hdd{Z}=(\vphi,\pi_1,\pi_2)$ and $\hdd{S}: \R^3\to\R$ given by $\hdd{S}(\hdd{Z})= \frac{1}{2}\pi_1^2-\frac{1}{2}\pi_2^2$. Then the system of equations can be rewritten as
\begin{align}\label{eq:polyKdel}
    K_1\partial_t\hdd{Z}(t,x)+K_2\partial_x\hdd{Z}(t,x)=\nabla\hdd{S}(\hdd{Z}(t,x)),
\end{align}
where 
\begin{align*}
     K_1 = \begin{pmatrix}0&-1&0\\1&0&0\\0&0&0\end{pmatrix} && K_2 = \begin{pmatrix}0&0&1\\0&0&0\\-1&0&0\end{pmatrix}.
\end{align*}
Note that \cref{eq:polyKdel} is somewhat similar to \cref{eq:symplJdt,eq:niekJdt}, but differs in a lot of aspects as well. An important difference is that the matrices $K_1$ and $K_2$ are no longer non-degenerate and therefore do not define complex structures. However, both matrices are still antisymmetric\footnote{This is why we introduced the extra minus signs on the last line of \cref{eq:hddsystem}.} which means that $\omega^{K_i}:=\langle\cdot, K_i\cdot\rangle$ does define a 2-form for $i=1,2$. Using these 2-forms, we rewrite \cref{eq:polyKdel} as
\begin{align}\label{eq:dS}
    (d\hdd{S})_{\hdd{Z}(t,x)} = \omega^{K_1}(\cdot,\partial_t\hdd{Z}(t,x))+\omega^{K_2}(\cdot,\partial_x\hdd{Z}(t,x)).
\end{align}

As $\omega^{K_1}=d\pi_1\wedge d\vphi$ and $\omega^{K_2}=-d\pi_2\wedge d\vphi$ are both degenerate forms on $\R^3$, this equation lives outside the world of symplectic geometry. However, the equation above does fit in a different geometric picture, called \textit{polysymplectic geometry}. We will very briefly discuss here how this works and leave the more detailed explanation of polysymplectic geometry to \cref{sec:polysympl}. 

\subsubsection*{The polysymplectic form}
First of all, define $\Omega_0=\omega^{K_1}\otimes\partial_t+\omega^{K_2}\otimes\partial_x$. This is a closed, non-degenerate $\R^2$-valued 2-form on $\R^3$, where $\R^2=T_tS^1\times T_xS^1$ is spanned by $\partial_t$ and $\partial_x$. By definition, $\Omega_0$ is called a \textit{polysymplectic form}. Note that $(d\hdd{Z})_{(t,x)}$ defines a linear map $\R^2\to\R^3$. Thus, for any vector $V\in\R^3$ the expression $\Omega_0(V,(d\hdd{Z})_{(t,x)}(\cdot))$ defines a linear automorphism of $\R^2$ and it therefore makes sense to consider the trace of this map. We denote this trace by $\Omega_0^\sharp((d\hdd{Z})_{(t,x)})(V)$. With this notation, \cref{eq:dS} is equivalent to  
\begin{align}\label{eq:polyform}
    (d\hdd{S})_{\hdd{Z}(t,x)} = \Omega_0^\sharp((d\hdd{Z})_{(t,x)}).
\end{align}
For a slightly more elaborate dicussion of $\Omega_0^\sharp$ and an explanation of why \cref{eq:dS,eq:polyform} are equivalent, see \cref{sec:polysympl} and in particular \cref{ex:polyHamEqns}.

Reformulating the wave equation in a covariant way as done above is usually called the De Donder-Weyl formulation of covariant Hamiltonian field theories. It is the standard way to alter the non-relativistic approach from the previous paragraph to a theory that is invariant under linear coordinate transformations in $O(1,1)$. The polysymplectic formulation of this approach can be found in many articles, such as \cite{gunther1987polysymplectic,helein2001hamiltonian,kanatchikov1993canonical}.\footnote{Some of these articles use multisymplectic instead of polysymplectic geometry to write down these equations. However, the two theories are very similar.}

In order to define a Floer theory for this equation, we must first define the action functional $\A_{\hdd{S}}:C^\infty(\T^2,\R^3)\to\R$. Note again that we are interested in periodic solutions, so we view $\hdd{Z}$ as a map on $\T^2$. Let $dV=dt\wedge dx$ denote the volume form on $\T^2$ and define
\begin{align}\label{eq:actionHdd}
    \A_{\hdd{S}}(\hdd{Z})=\int_{\T^2}\left(\pi_1(t,x)\partial_t\vphi(t,x)-\pi_2(t,x)\partial_x\vphi(t,x)-\hdd{S}(\hdd{Z}(t,x))\right)dV.
\end{align}
In \cref{sec:polysympl} it will be explained how this action functional is derived from the polysymplectic form $\Omega_0$. For now it is only important to note that critical points of this functional coincide with periodic solutions of \cref{eq:hddsystem} (see \cref{lem:critpoints}). To find these critical points, we again study smooth trajectories of the negative $L^2$-gradient of $\A_{\hdd{S}}$. These trajectories, or \textit{Floer curves}, are given by the equation
\begin{align}\label{eq:Floerhdd}
    \partial_s\wt{\hdd{Z}}(s,t,x)+K_1\partial_t\wt{\hdd{Z}}(s,t,x)+K_2\partial_x\wt{\hdd{Z}}(s,t,x)-\nabla\hdd{S}(\wt{\hdd{Z}}(s,t,x))=0,
\end{align}
where $\wt{\hdd{Z}}:\R\times \T^2\to\R^3$. \Cref{eq:Floerhdd} is the Floer curve corresponding to the system (\ref{eq:hddsystem}). 

Note that when $\wt{\hdd{Z}}$ is independent of both $t$ and $x$, the equation above reduces to the Morse theory of $\hdd{S}$. One might expect the equation to reduce to the original Floer \cref{eq:symplFloer} when $\wt{\hdd{Z}}$ is independent of $x$, however this is not the case as $K_1$ does not define a complex structure (compare with \cref{eq:BridgesFloer}). Also, when $S\equiv 0$ we would hope that the equation would satisfy a generalization of Gromov-Witten theory. However, the operator $K_1\partial_t+K_2\partial_x$ is too degenerate for some analogue of this theory to hold true. The next section will elaborate on this observation. 

\subsection{Degeneracy problem for Floer curves}\label{sub:degeneracy}
As mentioned in the discussion of Floer theory for classical mechanics, an analogue of lemma \ref{GW:zero} has to hold true in order for there to be any hope that this Floer theory will work out. As we saw in lemma \ref{GW:zero}, when $H\equiv 0$, Gromov's theory of pseudoholomorphic curves can be used to control the limit of the Floer curve for $s$ going to infinity. However, in our current model we will see that for $S\equiv 0$ problems occur.

The energy of a solution to \cref{eq:Floerhdd} is given by
\begin{align*}
    E(\wt{\hdd{Z}})=-\int_{-\infty}^{+\infty}\frac{d}{ds}\A_{\hdd{S}}(\wt{\hdd{Z}}_s)\,ds=\int_{-\infty}^{+\infty}\left(\int_{\T^2}|\partial_s\wt{\hdd{Z}}(s,t,x)|^2\,dV\right)\,ds.
\end{align*}
Let $\wt{\hdd{Z}}$ be a Floer curve for $\hdd{S}\equiv 0$ of finite energy. Assume that the mean of $\wt{\hdd{Z}}_s$ is zero for any $s$. We have to check if it follows that $\lim_{s\to+\infty}\wt{\hdd{Z}}_s=0$. 
\begin{counterexample}
Let 
\begin{align*}
\wt{\hdd{Z}}(s,t,x) = \begin{pmatrix}0\\ \cos(x)\\ 0\end{pmatrix}.
\end{align*}
Then $\partial_s\wt{\hdd{Z}}\equiv\partial_t\wt{\hdd{Z}}\equiv 0$ and $\partial_x\wt{\hdd{Z}}\in\ker K_2$, so indeed $\wt{\hdd{Z}}$ is a solution to \cref{eq:Floerhdd} for $S\equiv 0$. Also, $E(\wt{\hdd{Z}})=0$ as $\partial_s\wt{\hdd{Z}}\equiv 0$. The mean of $\wt{\hdd{Z}}_s$ is $\int_{\T^2}\wt{\hdd{Z}}_s(t,x)\,dV=0$, so $\wt{\hdd{Z}}$ satisfies our assumptions. However, the limit for $s$ going to $+\infty$ is clearly not the zero-function. 
\end{counterexample}
The reason that lemma \ref{GW:zero} does not hold for \cref{eq:Floerhdd}, is the fact that the operator $K_\partial=K_1\partial_t+K_2\partial_x$ has an infinite-dimensional kernel. It is given by $\ker K_\partial=\{(\phi,\partial_x\psi,\partial_t\psi)\mid \phi\textup{ is constant, }\psi:\T^2\to\R\}$. On the contrary, the kernel of the operator $J\partial_t$ from classical mechanics consists only of constant functions. The covariant Floer equation described above is therefore too degenerate to provide us with a working theory. That means that the De Donder-Weyl formulation of the wave equation is not suited for Floer theory, which explains why no such theory has been developed as of yet. The degeneracy problem of the operator $K_\partial$ has been pointed out already in \cite{bridgesTEA}. The next section is devoted to solving this problem. By the end of the section we will have found a new Floer equation that in fact is suitable for defining a Floer theory.


\section{Bridges regularization}\label{sec:Bridges}
This section is devoted to solving the degeneracy problem from \cref{sub:degeneracy}. Before discussing the general structure of a solution to that problem, we start hands-on by manipulating the equations in (\ref{eq:hddsystem}). The idea for the presented solution comes from \cite{bridgesTEA}.

The fundamental problem with \cref{eq:hddsystem} is that the matrices $K_1$ and $K_2$ are degenerate. To make them invertible, we add a new variable and a new equation to the system:
\begin{align}\begin{split}\label{eq:systembr}
    -\partial_t\pi_1(t,x) +\partial_x\pi_2(t,x) &= 0\\
    \partial_t\vphi(t,x) \textcolor{blue}{-\partial_x o(t,x)} &=\pi_1(t,x)\\
    \textcolor{blue}{\partial_t o(t,x)} -\partial_x\vphi(t,x)&=-\pi_2(t,x)\\
    \textcolor{blue}{-\partial_t \pi_2(t,x) +\partial_x \pi_1(t,x)} &=\textcolor{blue}{0}.
\end{split}\end{align}
If $Z=(\vphi,\pi_1,\pi_2,o)$ and $S(Z)=\frac{1}{2}\pi_1^2+\frac{1}{2}\pi_2^2$, then \cref{eq:systembr} can be formulated as
\begin{align}\label{eq:MiBridges}
    M_1\partial_tZ(t,x)+M_2\partial_xZ(t,x)=\nabla S(Z(t,x)),
\end{align}
where 
\begin{align*}
    M_1 = \begin{pmatrix}0&-1&0&\textcolor{blue}0\\1&0&0&\textcolor{blue}0\\0&0&0&\textcolor{blue}1\\\textcolor{blue}0&\textcolor{blue}0&\textcolor{blue}{-1}&\textcolor{blue}0\end{pmatrix} && M_2 = \begin{pmatrix}0&0&1&\textcolor{blue}0\\0&0&0&\textcolor{blue}{-1}\\-1&0&0&\textcolor{blue}0\\\textcolor{blue}0&\textcolor{blue}1&\textcolor{blue}0&\textcolor{blue}0\end{pmatrix}.
\end{align*}
The last rows and columns are made blue to highlight the fact that the matrices $K_1$ and $K_2$ are in some way "contained" in these new matrices. Note that the additional terms in \cref{eq:systembr} are chosen in such a way that $M_1$ and $M_2$ are both anti-symmetric invertible matrices. Moreover, they both define complex structures on $\R^4$. Thus, we can define two symplectic forms on $\R^4$ by
\begin{align*}
    \omega^{M_1}&:=\langle\cdot,M_1\cdot\rangle = d\pi_1\wedge d\vphi \textcolor{blue}{+ d\pi_2\wedge do}\\
    \omega^{M_2}&:=\langle\cdot,M_2\cdot\rangle = -d\pi_2\wedge d\vphi \textcolor{blue}{- d\pi_1\wedge do}.
\end{align*}
Clearly $\Omega_B:=\omega^{M_1}\otimes\partial_t+\omega^{M_2}\otimes\partial_x$ defines a closed, non-degenerate $\R^2$-valued 2-form on $\R^4$ and thus is a polysymplectic form. Just like in \cref{eq:polyform}, we can reformulate (\ref{eq:MiBridges}) as
\begin{align}\label{eq:bridgeswave}
    (dS)_{Z(t,x)}=\Omega_B^\sharp((dZ)_{(t,x)}).
\end{align}
We will refer to this equation as the Bridges formulation of the wave equation. \Cref{eq:bridgeswave} looks very similar to the De Donder-Weyl formulation of the wave equation, but the polysymplectic form used here has much more structure, as its separate components consist of symplectic forms. The idea is that this extra structure will allow us to prove an analogue of lemma \ref{GW:zero} and even of lemma \ref{GW:ndg} for this equation. 

To find the Floer equation corresponding to \cref{eq:bridgeswave}, we have to define an action functional once again. Again, we are interested in periodic solution, so we let $Z:\T^2\to\R^4$ and define the action functional\footnote{We dropped all the arguments $(t,x)$ in the notation here in order not to make the formula too long.}
\begin{align}\label{eq:actionBridges}
    \A_S(Z) = \int_{\T^2}\left(\pi_1\partial_t\vphi-\pi_2\partial_x\vphi-o(\partial_t\pi_2-\partial_x\pi_1)-\hdd{S}(\hdd{Z})\right)dV
\end{align}
whose critical points are the solutions to \cref{eq:bridgeswave}. We refer to \cref{sec:polysympl} for the derivation of this action functional and the computation of its critical points. The corresponding Floer equation for this action can be found once again by studying its negative $L^2$-gradient lines. These trajectories satisfy the equation 
\begin{align}\label{eq:BridgesFloer}
    \partial_s \wt{Z}(s,t,x) + M_1\partial_t \wt{Z}(s,t,x) + M_2\partial_x \wt{Z}(s,t,x) - \nabla S(\wt{Z}(s,t,x)) = 0,
\end{align}
where $\wt{Z}:\R\times \T^2\to\R^4$ is a smooth map. 

This new Floer equation is a lot nicer than \cref{eq:Floerhdd} in multiple ways. Some of the reasons why will be discussed in the following sections. Note that just as for \cref{eq:Floerhdd}, the equation above reduces to Morse theory when $\wt{Z}$ is independent of both $t$ and $x$. However, unlike \cref{eq:Floerhdd}, when $\wt{Z}$ is just independent of $x$, the original Floer \cref{eq:symplFloer} is recovered. This is due to the fact that $M_1$ defines a complex structure on $\R^4$. Moreover, as $M_2$ defines a complex structure as well, the same statement holds when $\wt{Z}$ is only independent of $t$. Finally, when $S\equiv 0$ we no longer have the degeneracy problem mentioned in \cref{sub:degeneracy} and the discussion right before. The next section will show that indeed an analogue of lemma \ref{GW:zero} can be proven for \cref{eq:BridgesFloer}.

\subsection{Degeneracy problem for Floer curves - solved}\label{sub:degsolved}
From this section on, we must explicitly define the periodicity conditions that we are interested in. We will always assume the space-period to be $2\pi$ and the time-period to be some real number $T$. That is, we identify the 2-torus with $\T^2=(\R/T\Z)\times(\R/2\pi\Z)$. It must be mentioned that the degeneracy problem depends on the ratio $\ratio=\frac{2\pi}{T}$. Note, for example, when $T=2\pi$ that the free wave equation (\ref{eq:wave}) has an infinite-dimensional space of solutions containing $\vphi(t,x)=f(t+x)+g(t-x)$ for all $2\pi$-periodic functions $f$ and $g$. This type of degeneracy is different from the degeneracy problem of \cref{sub:degeneracy} in the sense that it is already a part of the free wave equation and does not depend on the geometric approach that we choose to take. Even in the infinite-dimensional symplectic approach taken in \cite{paperoliverniek} (see \cref{sub:niekapproach}) this type of degeneracy is excluded a priori. Thus, we restrict our attention to the case where $\ratio$ has \textit{irrationality measure} equal to 2.
\begin{definition}\label{def:irrmeasure}
The \emph{irrationality measure} of a real number $r$ is defined to be the infimum of all $\rho$ for which there exists some constant $c$ such that 
\begin{align*}
\frac{c}{q^\rho}<|r-\frac{p}{q}|
\end{align*}
for all $\frac{p}{q}\in\Q$.
\end{definition}

Theorem E.3 from \cite{diophantine} asserts that the set of real number with irrationality measure equal to 2 has full measure within the set of real numbers. Thus, the requirement on $\ratio$ stated above is not too restrictive. Note that in particular, our assumption implies that $\ratio$ is irrational. 

Now that we have settled the necessary condition on $\ratio$, we will prove an analogue of lemma \ref{GW:zero}. First of all, we define the \textit{energy} of a solution $\wt{Z}:\R\times \T^2\to\R^4$ of \cref{eq:BridgesFloer} as 
\begin{align*}
  E(\wt{Z}) &= -\int_{-\infty}^{+\infty}\frac{d}{ds}\A_S(\wt{Z}(s))\,ds\\
    &=\int_{-\infty}^{+\infty}\left(\int_{\T^2}\abs{\partial_s \wt{Z}}^2dV\right)ds.
\end{align*}

\begin{lemma}\label{lem:BridgesGWzero}
Let $\wt{Z}$ be a Floer curve of finite energy for $S\equiv 0$. Moreover, assume that $\wt{Z}_s$ has mean zero for every $s\in\R$. Then
\begin{align*}
\lim_{s\to\infty}\wt{Z}_s=0
\end{align*}
in $C^\infty(\T^2,\R^4)$.
\end{lemma}
\begin{proof}
We are going to prove the lemma by showing that the $L^2$-norms of all space-time derivatives of $\wt{Z}_s$ converge to 0. This will imply convergence in $C^\infty(\T^2,\R^4)$.
As $\wt{Z}_s$ is periodic in time and space we can define its Fourier transform as follows:
\[
\wt{Z}_s(t,x) = \sum_{m,k\in \Z}\hat{Z}_s(m,k)e^{\frac{2\pi i}{T}mt}e^{ikx}.
\]
Note that the assumption that the mean of $\wt{Z}_s$ is zero for all $s$, implies that $\hat{Z}_s(0,0)=0$. Thus, in what follows we may restrict our attention to $(m,k)\neq(0,0)$.
Applying equation (\ref{eq:BridgesFloer}) to the Fourier series gives
\begin{align}\label{eq:FourierFloer}
    0&=\partial_s\hat{Z}_s(m,k) + \frac{2\pi i}{T}mM_1\hat{Z}_s(m,k)+ikM_2\hat{Z}_s(m,k)\\
    &=\partial_s\hat{Z}_s(m,k) +A(m,k)\hat{Z}_s(m,k),\nonumber
\end{align}
where 
\begin{align*}
A(m,k)&=\frac{2\pi i}{T}mM_1+ikM_2\\
&=\begin{pmatrix}
0 &- \ratio im & ik & 0\\
\ratio im & 0 & 0 & -ik\\
-ik & 0 & 0 & \ratio im\\
0 & ik & -\ratio im & 0
\end{pmatrix}.
\end{align*}
As $A(m,k)$ is a Hermitian matrix for every $m$ and $k$, its eigenvalues are real. The eigenvalues $\{\lambda^1_+(m,k),\lambda^1_-(m,k),\lambda^2_+(m,k),\lambda^2_-(m,k)\}$ are given by 
\begin{align*}
\lambda^1_+(m,k) &=-\lambda^1_-(m,k) = \abs{k+m\ratio}>0\\
\lambda^2_+(m,k) &=-\lambda^2_-(m,k) = \abs{k-m\ratio}>0.
\end{align*}
Note that the strict inequalities above come from the assumption that $\ratio$ is irrational and $(m,k)\neq(0,0)$. 

\Cref{eq:FourierFloer} is a linear ODE and has solution
\begin{align*}
\hat{Z}_s(m,k) = e^{-A(m,k)s}\hat{Z}_0(m,k)
\end{align*}
for some initial condition $\hat{Z}_0(m,k)$. Thus,
\begin{align}\label{eq:dsZFourier}
\partial_s\hat{Z}_s(m,k) = -A(m,k)e^{-A(m,k)s}\hat{Z}_0(m,k). 
\end{align}
If $v^i_+(m,k)$ and $v^i_-(m,k)$ denote eigenvectors corresponding to $\lambda^i_+(m,k)$ and $\lambda^i_-(m,k)$ respectively, than we can write 
\[
\hat{Z}_0(m,k)=\sum_{i=1}^2 \left(\alpha^i_+(m,k)v^i_+(m,k)+\alpha^i_-(m,k)v^i_-(m,k)\right).
\]
Note that as $A(m,k)$ is Hermitian, the different eigenvectors are orthogonal. Also
\[E(\wt{Z})=\int_{-\infty}^{+\infty}\left(\int_{\T^2}\abs{\partial_s \wt{Z}_s(t,x)}^2dV\right)\,ds=\int_{-\infty}^{+\infty}\left(\sum_{m,k}\abs{\partial_s\hat{Z}_s(m,k)}^2\right)\,ds<+\infty,\]
which implies in particular that there is a sequence $s_j\to+\infty$ such that $\partial_s\hat{Z}_{s_j}(m,k)\to 0$. \Cref{eq:dsZFourier} shows that this is only possible when $\alpha^i_-(m,k)=0$ for $i=1,2$. We see that 
\begin{align}\label{eq:hatZformula}
\hat{Z}_s(m,k)=e^{-\lambda^1_+(m,k)s}\alpha^1_+(m,k)v^1_+(m,k)+e^{-\lambda^2_+(m,k)s}\alpha^2_+(m,k)v^2_+(m,k).
\end{align}

Fix $l_1,l_2\in\N$ and let $\epsilon>0$. We want to prove that for $s$ big enough it holds that $||\partial_t^{l_1}\partial_x^{l_2}\wt{Z}_s||_{L^2}^2<\epsilon$. Note first that the smoothness of $\wt{Z}_0$ implies that 
\begin{align*}
    ||\partial_t^{l_1}\partial_x^{l_2}\wt{Z}_0||_{L^2}^2 = \sum_{m,k}\sum_{i=1}^2 (\ratio m)^{2l_1}k^{2l_2}|\alpha^i_+(m,k)v^i_+(m,k)|^2 <+\infty.
\end{align*}
Thus, in particular we can find some $N$ such that 
\begin{align*}
    \sum_{|(m,k)|>N}\sum_{i=1}^2 (\ratio m)^{2l_1}k^{2l_2}|\alpha^i_+(m,k)v^i_+(m,k)|^2 <\frac{\epsilon}{2}.
\end{align*}
As $\lambda^i_+(m,k)>0$ for all $m,k$, this implies that for any $s>0$
\begin{align}\label{eq:epsi2}
    \sum_{|(m,k)|>N}\sum_{i=1}^2 (\ratio m)^{2l_1}k^{2l_2}e^{-2\lambda^i_+(m,k)s}|\alpha^i_+(m,k)v^i_+(m,k)|^2 <\frac{\epsilon}{2}.
\end{align}
Now, let $\lambda:=\min\{\lambda^i_+(m,k)\mid i=1,2\textup{ and }|(m,k)|\leq N\}$ and note $\lambda>0$. Then
\begin{align*}
    \sum_{|(m,k)|\leq N}\sum_{i=1}^2 (\ratio m)^{2l_1}k^{2l_2}e^{-2\lambda^i_+(m,k)s}|\alpha^i_+(m,k)v^i_+(m,k)|^2 \leq e^{-\lambda s}||\partial_t^{l_1}\partial_x^{l_2}\wt{Z}_0||_{L^2}^2.
\end{align*}
Thus, if we choose $s_0>\frac{1}{\lambda} \log(2||\partial_t^{l_1}\partial_x^{l_2}\wt{Z}_0||_{L^2}^2/\epsilon)$ (assuming $\epsilon$ is small enough for this to be positive), then for $s>s_0$ we get that 
\begin{align*}
    \sum_{|(m,k)|\leq N}\sum_{i=1}^2 (\ratio m)^{2l_1}k^{2l_2}e^{-2\lambda^i_+(m,k)s}|\alpha^i_+(m,k)v^i_+(m,k)|^2 < \frac{\epsilon}{2}.
\end{align*}
Combining this with \cref{eq:epsi2} it follows that for $s>s_0$
\begin{align*}
    ||\partial_t^{l_1}\partial_x^{l_2}\wt{Z}_s||_{L^2}^2 =  \sum_{m,k}\sum_{i=1}^2 (\ratio m)^{2l_1}k^{2l_2}e^{-2\lambda^i_+(m,k)s}|\alpha^i_+(m,k)v^i_+(m,k)|^2 < \epsilon.
\end{align*}
This proves the lemma
\end{proof}

This proof shows that \cref{eq:BridgesFloer} solves the degeneracy problem from \cref{sub:degeneracy}. In \cref{sec:particlefield} we will prove that even an analogue of lemma \ref{GW:ndg} can be proven for the new Floer equation.

\subsection{Underlying structure}\label{sub:underlyingstr}
At the start of \cref{sec:Bridges}, a new variable and equation were introduced in order to improve the Floer equation corresponding to De Donder-Weyl. However, even though it turns out that indeed we can get an improved Floer equation out of this, no explanation was given on where this new set of equations comes from. This section is devoted to explaining the structure behind \cref{eq:systembr}. The ideas from this section are based on \cite{bridgesTEA}.

First, we look back at our discussion of classical mechanics in \cref{sec:CMtoFT}. The fundamental operator here was $J\frac{d}{dt}$. This operator comes in naturally also from a geometric point of view, when we look at differential forms on $\R$. The most fundamental operator on differential forms is the exterior derivative $d$. For functions $q$ on $\R$ it is given by $dq=\frac{d}{dt}q\,dt$. Correspondingly, one can also define the \textit{codifferential} $\delta$ that takes a 1-form to a 0-form on $\R$. It is given by $\delta(p\,dt)=-\frac{d}{dt}p$ for functions $p$ on $\R$. Choosing the basis $\{1,dt\}$ for the total exterior algebra $\bigwedge(T^*\R)$ identifies it with $\R^2$. Under this identification the operator $J\frac{d}{dt}$ on functions $u:\R\to\R^2$ corresponds to the operator 
\begin{align}\label{eq:Jdel2}
\begin{pmatrix}
o&\delta\\d&0
\end{pmatrix}
\end{align}
on sections $q+p\,dt:\R\to \bigwedge(T^*\R)$.

Moving from classical mechanics to field theory, a space-variable is added to the system and the functions $\vphi$ are now defined on $\R^2$. We can try to follow the same idea as above to get an operator on differential forms on $\R^2$. To define the codifferential, first a metric has to be chosen on $\R^2$. Since the wave equation comes from the Laplace operator on Minkowski space, we put a Minkowski metric on $\R^2$. We choose the Minkowski metric of signature $(+-)$, meaning that $\langle\partial_t,\partial_t\rangle = -\langle\partial_x,\partial_x\rangle = 1$. We will denote $\R^2$ with this Minkowski metric by $\R^{1,1}$. The total exterior algebra bundle $\Omega(\R^{1,1}) = \Gamma(\bigwedge(T^*\R^{1,1}))$ of this space comes equipped with a codifferential $\delta_k$ that transforms $k$-forms into $(k-1)$-forms. It is defined by $\delta_k = (-1)^k\star^{-1}d\star$, where $\star$ denotes the Hodge star. The latter is the linear operator given by 
\begin{align*}
    \star 1 &= dV\\
    \star dt&= dx\\
    \star dx &= dt\\
    \star dV &= -1,
\end{align*}
A simple calculation shows
\begin{align*}
    \delta_1(\pi_1dt+\pi_2dx) &= -\partial_t \pi_1+\partial_x \pi_2\\
    \delta_2(o\,dV)&= -\partial_t odx -\partial_x odt,
\end{align*}
for functions $\pi_1$,$\pi_2$ and $o$. Now, analogously to the operator (\ref{eq:Jdel2}) defined above, Bridges defines an operator $J_\partial$ on $\Omega(\R^{1,1})$ by
\begin{align*}
    J_\partial = \begin{pmatrix}0&\delta_1&0\\d_0&0&\delta_2\\0&d_1&0\end{pmatrix}.
\end{align*}

The bundle $\Omega(\R^{1,1})$ can be trivialised by choosing the global frame $\{1,dt,dx,dV\}$. In this identification, $\Omega(\R^{1,1})$ can be seen as the space of maps $\R^{1,1}\to \R^{2,2}$, where we use the Lorentzian space\footnote{Sometimes we will just denote it by $\R^4$, when the focus does not lie on the metric.} $\R^{2,2}$ because the induced metric on the total exterior algebra is given by 
\begin{align*}
    \mathcal{L} = \begin{pmatrix}1&0&0&0\\0&1&0&0\\0&0&-1&0\\0&0&0&-1\end{pmatrix}.
\end{align*}
in this basis. We can also write out $J_\partial$ as 
\begin{align*}
    J_\partial = J_1\partial_t + J_2\partial_x,
\end{align*}
where 
\begin{align*}
    J_1 = \begin{pmatrix}0&-1&0&0\\1&0&0&0\\0&0&0&-1\\0&0&1&0\end{pmatrix} && J_2=\begin{pmatrix}0&0&1&0\\0&0&0&-1\\1&0&0&0\\0&-1&0&0\end{pmatrix}.
\end{align*}

Now, Bridges looks at Hamiltonians $S:\R^{2,2}\to \R$ and considers the equation:
\begin{align}\label{eq:Jdel}
    J_\partial Z(t,x) = \nabla_LS(Z(t,x)),
\end{align}
for $Z\in \Omega(\R^{1,1})$. Here, $\nabla_L=\mathcal{L}\nabla$ is the Lorentzian gradient. After multiplying both sides of the equation with the Lorentzian metric $\mathcal{L}$, equation (\ref{eq:Jdel}) can be written as
\begin{align*}
    M_1\partial_t Z(t,x) + M_2\partial_x Z(t,x) = \nabla S(Z(t,x)),
\end{align*}
where $\nabla$ is the Riemannian gradient. This is exactly \cref{eq:MiBridges}. Note that it is equivalent to \cref{eq:Jdel} as $\mathcal{L}$ is invertible. 

The exposition above shows that \cref{eq:MiBridges} arises naturally from a geometric point of view as a generalisation of \cref{eq:symplJdt}. Note that if we would have restricted the operator $J_\partial$ to the space of 0-forms and 1-forms, we would have recovered \cref{eq:polyKdel}. However, when we are working on a 2-dimensional space-time then excluding the 2-forms is somehow less logical than considering $J_\partial$ as an operator on the full space of differential forms. Even though in this article we focus primarily on 2-dimensional space-time, the construction of $J_\partial$ can easily be generalized to higher dimensions. See \cite{bridgesTEA} for the details. 

\subsection{The infinite-dimensional viewpoint}\label{sub:infvp}
As discussed above, \cref{eq:BridgesFloer} is in some way more natural to consider as a Floer equation than \cref{eq:Floerhdd}. In this section we will discuss another advantage of \cref{eq:BridgesFloer} over \cref{eq:Floerhdd}. That is that, after making a fixed splitting of time and space, \cref{eq:BridgesFloer} fits into the infinite-dimensional symplectic framework of \cite{paperoliverniek}. 

As discussed in \cref{sub:niekapproach}, \cite{paperoliverniek} analyse the wave equation by using infinite-dimensional symplectic geometry. To translate to their viewpoint, for the rest of this section we regard $Z$ as a map $S^1\to L^2(S^1,\R^4)$. By doing this we break the covariance and view $Z$ as a map from time to a space of functions of the space-variable. In order to get a symplectic equation, we first need a complex structure on $L^2(S^1,\R^4)$. Note that $M_1^2=-\Id$, so it provides us with this complex structure. As we know from \cref{eq:niekJdt}, the symplectic equation corresponding to some Hamiltonian $\niek{H}:L^2(S^1, \R^4)\to \R$, is
\begin{align}\label{eq:complexM}
M_1\frac{d}{dt}Z(t)=\grad\niek{H}(Z(t)).
\end{align}
In order to recover \cref{eq:MiBridges}, we define
\[
\niek{H}(Z) = -\frac{1}{2}\int_{S^1}\langle M_2\partial_x Z,Z\rangle\,dx+\int_{S^1} S(Z(x))\, dx.
\]
Then we see that $\grad\niek{H}(Z) = -M_2\partial_x Z+S\circ Z$, so that \cref{eq:complexM} becomes the same as \cref{eq:MiBridges}. Note in particular that the operator $M_2\partial_x$ that was used to define this Hamiltonian is self-adjoint with respect to the $L^2$-inner product. This ensures the existence of a complete eigenbasis with real eigenvalues, which is the assumption that \cite{paperoliverniek} starts with (see \cref{sub:existence}). 

We see that indeed the equations coming from the Bridges regularization can be put into the framework of \cite{paperoliverniek}. This was not possible for the original De Donder-Weyl equation, as neither $K_1$ nor $K_2$ define a complex structure. The upshot is that for certain choices of non-linearities, results from \cite{paperoliverniek} may be used to conclude the existence of Floer curves and thus the existence of periodic solutions of the wave equation. In \cref{sub:existence} we will encounter a natural class of non-linearities that become \emph{$\infty$-regularizing} (see \cref{def:reg}) when translated to the infinite-dimensional framework and thus fit into the work of Fabert and Lamoree. However, as we will see in \cref{sec:particlefield}, the covariant Floer equation is also more general in a certain way and we can introduce non-linearities that the framework of \cite{paperoliverniek} cannot deal with.

\subsection{The space of solutions}\label{sub:spaceofsol}
Before concluding this section, we want to see that the extra variable $o$ in \cref{eq:systembr} does not alter the space of solutions that we consider, in the sense that the space of solutions should still be in one-to-one correspondence with the solutions to the wave equation. It turns out that this is not entirely true, but every solution to the wave equation yields a one-dimensional space of solutions to \cref{eq:systembr}. This problem can easily be fixed. To make the discussion a little more general, we introduce a non-linearity to the equation now. That is, we consider the wave equation
\begin{align}\label{eq:nonlinwave}
    -\partial_t^2\vphi(t,x)+\partial_x^2\vphi(t,x) = f(\vphi(t,x)),
\end{align}
where, for now, $f$ can be any function. The system of equations corresponding to \cref{eq:systembr}, including this non-linearity is
\begin{align}\begin{split}\label{eq:systembrnonlin}
    -\partial_t\pi_1(t,x) +\partial_x\pi_2(t,x) &= f(\vphi(t,x))\\
    \partial_t\vphi(t,x) {-\partial_x o(t,x)} &=\pi_1(t,x)\\
    {\partial_t o(t,x)} -\partial_x\vphi(t,x)&=-\pi_2(t,x)\\
    {-\partial_t \pi_2(t,x) +\partial_x \pi_1(t,x)} &={0}.
\end{split}\end{align}
Clearly, when $Z=(\vphi,\pi_1,\pi_2,o)$ is a solution to \cref{eq:systembrnonlin}, then $\vphi$ solves the wave equation (\ref{eq:nonlinwave}). This can be seen by filling in $\pi_1$ and $\pi_2$ from the second and third equation into the first equation in (\ref{eq:systembrnonlin}). For the converse, note that filling in $\pi_1$ and $\pi_2$ into the last equation, yields that $o$ solves the free wave equation. We will use the following lemma. 
\begin{lemma}
The only periodic solutions of the free wave equation 
\begin{align*}
 {-\partial_t^2 o(t,x) +\partial_x^2 o(t,x)} &={0}
\end{align*}
are $o\equiv c$ for some constant $c\in\R$, when $\ratio=\frac{2\pi}{T}$ is irrational.
\end{lemma}
\begin{proof}
Clearly, all constant maps $o\equiv c$ satisfy the wave equation. Conversely, if $o$ satisfies the free wave equation, then its Fourier transform $\hat{o}$ satisfies
\begin{align*}
\ratio^2m^2\hat{o}(m,k) -k^2 \hat{o}(m,k) = 0.
\end{align*}
As $\ratio^2m^2 - k^2$ is zero only when $m=k=0$, it must hold that $o(m,k)=0$ for $(m,k)\neq(0,0)$. Thus the only non-zero Fourier coefficient is $o(0,0)$, which proves that $o$ is constant. 
\end{proof}
Now, let $\vphi$ be any solution to \cref{eq:nonlinwave}. Then setting $\pi_1 = \partial_t\vphi$, $\pi_2=\partial_x\vphi$ and $o\equiv c$ yields a solution to \cref{eq:systembrnonlin}. Thus indeed, we find that every solution of the wave equation corresponds to a one-dimensional space of solutions to \cref{eq:systembrnonlin}. We can circumvent this problem by requiring $o$ to have mean zero. 

One might wonder what the actual difference between \cref{eq:hddsystem} and \cref{eq:systembr} is, now that we know that the extra variable $o$ is constant anyways. The difference however doesn't lie per se in the equation itself, but in the Floer equation it results in. To stress the difference between the two Floer equations we write them out once more in a different way. Let $\epsilon$ be a parameter and consider the set of equations
\begin{align*}
    \partial_s\wt\vphi(s,t,x)  -\partial_t\wt\pi_1(s,t,x) + \partial_x\wt\pi_2(s,t,x) &= 0\\
    \partial_s\wt\pi_1(s,t,x)  +\partial_t\wt\vphi(s,t,x) - \eps \partial_x \wt{o}(s,t,x) &=\wt\pi_1(s,t,x)\\
    \partial_s\wt\pi_2(s,t,x)  +\eps\partial_t \wt{o}(s,t,x) - \partial_x\wt\vphi(s,t,x) &=-\wt\pi_2(s,t,x)\\
    \eps\partial_s \wt{o}(s,t,x)  -\eps\partial_t \wt\pi_2(s,t,x) + \eps\partial_x \wt\pi_1(s,t,x) &= 0.
\end{align*}
For $\epsilon=0$ this gives the Floer \cref{eq:Floerhdd} for de De Donder-Weyl equation from \cref{sub:relFT}, while for $\epsilon=1$ it describes the Floer \cref{eq:BridgesFloer} corresponding to Bridges' equations. Even though the $o$ variable is constant for periodic solutions of \cref{eq:systembr}, it doesn't have to be for solutions of the Floer equation. As should hopefully be clear by now, the change from $\epsilon=0$ to $\epsilon=1$ in the equations above, precisely turns the Floer equation into a workable equation upon which we can build the theory.


\section{The polysymplectic formalism}\label{sec:polysympl}
In \cref{sec:CMtoFT,sec:Bridges} the language of polysymplectic geometry was briefly introduced and used to formulate symmetric approaches to field theory. This section will give the more rigorous background of the material and show exactly how the polysymplectic equations come about. As a reference for this section we refer to \cite{gunther1987polysymplectic}. 

\subsection{The polysymplectic Hamiltonian formalism}
Let $M$ be a vector space. Most of the definitions in this section apply to more general manifolds as well, but as mentioned in the introduction, the treatment of this will be postponed to a follow-up article.
\begin{definition}
Let $\Omega$ be an $\R^n$-valued form on $M$. If $\Omega(V,\cdot)=0$ implies $V=0$ for $V\in TM$, then $\Omega$ is called \emph{non-degenerate}.
\end{definition}
\begin{definition}\label{def:polyform}
An $\R^n$-valued 2-form $\Omega$ on $M$ is called \emph{polysymplectic} if it is closed and non-degenerate. The pair $(M,\Omega)$ is called a \emph{polysymplectic manifold}.
\end{definition}

We know that a symplectic form combined with a Hamiltonian function yields an equation, like the one in \cref{eq:sympl}. In a similar way we want to combine polysymplectic forms with Hamiltonians. We already saw two examples in \cref{eq:polyform,eq:bridgeswave}. Here, the general construction will be explained. 

Let $X: \R^n\to TM$ denote any linear map and $V\in TM$. Then the map $\nu\mapsto \Omega(V,X(\nu))$ is a linear map from $\R^n$ to itself of which we can take the trace. Thus we get a map $TM\to \R$ given by $V\mapsto \textup{tr }\Omega(V,X(\cdot))$ and denote this map by $\Omega^\sharp(X)$. Thus $\Omega^\sharp:\textup{Hom}(\R^n,TM)\to T^*M$ is given by
\[
\Omega^\sharp(X)(V)=\textup{tr}\left(\nu\mapsto\Omega(V,X(\nu))\right).
\]

Now for any map $Z:\R^n\to M$ it holds that $(dZ)_{\vec{x}}$ is a linear map $\R^n\to TM$ for $\vec{x}\in\R^n$. Given a Hamiltonian\footnote{Note that polysymplectic Hamiltonians are denoted by $S$ in this article, whereas symplectic Hamiltonians are denoted $H$.} $S:M\to \R$ we get the Hamiltonian equation
\begin{align}\label{eq:generalpoly}
    (dS)_{Z(\vec{x})} = \Omega^\sharp((dZ)_{\vec{x}})
\end{align}
for any $\vec{x}\in \R^n$

\begin{example}
First of all, let's take $n=1$, $M=\R^2$ and $\Omega=\omega=dp\wedge dq$ the standard symplectic form on $\R^2$. For a curve $u:\R\to\R^2$ we get that $\omega^\sharp((du)_t)(V)=\textup{tr }\omega(V,(du)_t(\cdot))=\omega(V,\frac{d}{dt}u(t))$. Thus, \cref{eq:generalpoly} reduces to \cref{eq:sympl} in this case. It follows that the polysymplectic formalism indeed extends the symplectic formalism. 
\end{example}
\begin{example}\label{ex:polyHamEqns}
Now let $n=2$, $M=\R^d$ and choose 2-forms $\omega_1$ and $\omega_2$ on $\R^d$ such that $\Omega=\omega_1\otimes\partial_t+\omega_2\otimes\partial_x$ is a polysymplectic form. For $Z:\R^2\to\R^d$  the map $\nu\mapsto\Omega(V,(dZ)_{(t,x)}(\nu))$ is given by
\begin{align*}
\begin{pmatrix}\nu_t\\\nu_x\end{pmatrix}\mapsto\begin{pmatrix}\omega_1(V,\partial_tZ(t,x))&\omega_1(V,\partial_xZ(t,x))\\\omega_2(V,\partial_tZ(t,x))&\omega_2(V,\partial_xZ(t,x))\end{pmatrix}\begin{pmatrix}\nu_t\\\nu_x\end{pmatrix}
\end{align*}
for $\nu=(\nu_t,\nu_x)\in\R^2$.
This map has trace $\omega_1(V,\partial_tZ(t,x))+\omega_2(V,\partial_xZ(t,x))$. Thus,
\[\Omega^\sharp((dZ)_{(t,x)})=\omega_1(\cdot,\partial_tZ(t,x))+\omega_2(\cdot,\partial_xZ(t,x)).\]
If we put $d=3$ and $\omega_i=\omega^{K_i}$, then \cref{eq:generalpoly} becomes \cref{eq:polyform} and the above computation shows that this is indeed equivalent to \cref{eq:dS}. For $d=4$ and $\omega_i=\omega^{M_i}$ \cref{eq:bridgeswave} is recovered. 
\end{example}
 \begin{remark}
 Note that \cref{def:polyform} of a polysymplectic form states that the form $\Omega$ has to be non-degenerate. When $\Omega$ can be written as $\Omega=\sum\omega_i\otimes\partial_{x_i}$ this does not however imply that the separate forms $\omega_i$ have to be non-degenerate. Indeed, this is exactly where the difference between $\Omega_0$ from \cref{sub:relFT} and $\Omega_B$ from \cref{sec:Bridges} lies. The components $\omega^{K_i}$ of $\Omega_0$ are degenerate forms on $\R^3$, yet $\Omega_0$ itself is non-degenerate. On the other hand, the components $\omega^{M_i}$ of $\Omega_B$ are themselves already non-degenerate forms on $\R^4$, from which it follows that $\Omega_B$ is non-degenerate as well. 
 \end{remark}
 
 \subsection{From polysymplectic forms to action functionals}\label{sub:polyactions}
 Now that we have established how polysymplectic forms combine with Hamiltonians to yield field equations, we want to construct the action functional that will allow us to construct a Floer theory. Thus, given a polysymplectic form $\Omega$ on $M$ with values in $\R^n$ and a Hamiltonian $S:M\to\R$, we want to construct a functional $\A_S^\Omega:C^\infty(\T^n,M)\to\R$, such that the critical points of this functional coincide with the solutions to \cref{eq:generalpoly}. In symplectic geometry the action is defined by pulling back the symplectic form along a periodic map and integrating over the circle. As we are working with $\R^n$-valued forms we cannot simply integrate the pullback. First, we must construct an ordinary differential form out of the polysymplectic form. 
 
 Since we are interested in periodic solutions to \cref{eq:generalpoly}, we will actually replace $\R^n$ by $\T^n$ and consider $\Omega$ as a form with values in $T\T^n$ instead of $\R^n$. By contracting with the volume form $dV=dx_1\wedge\cdots\wedge dx_n$ on $\T^n$ it can also be viewed as  a $\bigwedge^{n-1}(T^*\T^n)$-valued form. As $T\T^n$ is a trivial bundle, we can write $\Omega=\sum_i \omega_i\otimes \partial_i$, for closed 2-forms $\omega_i$ on $M$ and where $\partial_i:=\partial_{x_i}$. After contraction with $dV$ it becomes $\sum_i(-1)^{i+1}\omega_i\otimes dx_1\wedge\cdots\wedge\widehat{dx_i}\wedge\cdots\wedge dx_n$. Define $\wt{\Omega}:=\sum_i (-1)^{i+1}\omega_i\wedge dx_1\wedge\cdots\wedge\widehat{dx_i}\wedge\cdots\wedge dx_n$ as an $(n+1)$-form on $\T^n\times M$.\footnote{This interchanging of tensor products and wedge products seems ad hoc, but comes from a bigger geometric picture. In the language of \cite{multipoly_forgergomes}, $\Omega$ is the \emph{symbol} of the horizontal form $\wt{\Omega}$. As we are dealing with linear spaces, the construction boils down to the simple change of tensor products and wedge products. We refer to \cite{multipoly_forgergomes} for the general construction.} 
 
 As $M$ is a vector space, all the forms $\omega_i$ are exact and we can write $\omega_i=d\theta_i$. This implies that $\wt{\Omega}=d\wt{\Theta}$ for $\wt{\Theta}=\sum_i (-1)^{i+1}\theta_i\wedge dx_1\wedge\cdots\wedge\widehat{dx_i}\wedge\cdots\wedge dx_n$. Note that every map $Z:\T^n\to M$ can be seen as a section $Z:\T^n\to \T^n\times M$ of the trivial bundle over $\T^n$. Thus, it holds that $Z^*\wt{\Theta}$ is a well-defined $n$-form on $\T^n$. Therefore, we can define the action functional $\A_S^\Omega:C^\infty(\T^n,M)\to\R$ by  
 \begin{align*}
     \A_S^\Omega(Z) = \int_{\T^n}Z^*\wt{\Theta}-\int_{\T^n}S(Z(t,x))\,dV
 \end{align*}
\begin{remark}
Note that there can be multiple choices of primitives $\theta_i$ leading to different action functionals. However, if $\wt{\Theta}$ and $\wt{\Psi}$ are two choices of primitives $\wt{\Omega}=d\wt{\Theta}=d\wt{\Psi}$, then $d(\wt{\Theta}-\wt{\Psi})=0$ so that $Z^*(\wt{\Theta}-\wt{\Psi})\in H^n_{dR}(\T^n)$. Thus $Z^*(\wt{\Theta}-\wt{\Psi})=const\cdot dV$, meaning that $\A_S^\Omega$ just shifts by a constant. As we are only interested in relative values of $\A_S^\Omega$ this does not matter.
\end{remark}
\begin{example}
If we take $n=2$ and $\Omega$ to be $\Omega_0$ from \cref{sub:relFT}, we get $\wt{\Omega}_0=d\pi_1\wedge d\vphi\wedge dx+d\pi_2\wedge d\vphi\wedge dt$ and $\wt{\Theta}_0=\pi_1 d\vphi\wedge dx+\pi_2 d\vphi\wedge dt$. Thus, $Z^*\wt{\Theta}_0=(\pi_1\partial_t\vphi-\pi_2\partial_x\vphi)dV$, so that indeed the action functional $\A_S^\Omega$ defined above coincides with \cref{eq:actionHdd}. A similar computation shows that for $\Omega=\Omega_B$ the action functional from \cref{eq:actionBridges} is recovered.
\end{example}

 \begin{lemma}\label{lem:critpoints}
 The critical points of $\A_S^\Omega$ coincide with the solutions of \cref{eq:generalpoly}.
 \end{lemma}
 \begin{proof}
 Let $Y$ be a tangent vector to $Z\in C^\infty(\T^n,M)$. We view $Y$ as a map $\T^n\to TM$. Now 
 \begin{align}\label{eq:dA}
     (d\A_S^\Omega)_Z(Y) &= \int_{\T^n} Z^*\Lie_{Y(\vec{x})}\wt{\Theta}-\int_{\T^n} (dS)_{Z(\vec{x})}(Y(\vec{x}))\,dV.
 \end{align}
 To work out the first term, notice that $\T^n$ has no boundary, so that 
 \[\Lie_{Y(\vec{x})}\wt{\Theta}=\iota_{Y(\vec{x})}d\wt{\Theta}=\iota_{Y(\vec{x})}\wt{\Omega}.\] So \[Z^*\Lie_{Y(\vec{x})}\wt{\Theta}=\sum_i\omega_i(Y(\vec{x}),\partial_i Z(\vec{x}))dV.\] 
 Filling this back into \cref{eq:dA} yields that $d\A_S^\Omega=0$ precisely when
 \begin{align*}
 (dS)_{Z(\vec{x})}=\sum_i\omega_i(\cdot,\partial_i Z(\vec{x}))
 \end{align*}
 for all $\vec{x}\in \T^n$. As \cref{ex:polyHamEqns} shows, this is equivalent to \cref{eq:generalpoly} for $n=2$. For general $n$ a similar computation applies. 
 \end{proof}


\section{Prime example: coupled particle-field systems}\label{sec:particlefield}
As an important example of our theory we consider a mechanical system coupled to a field, i.e. a symplectic system coupled to a polysymplectic one. This provides an interesting example from a physicists point of view\footnote{Coupled particle-field systems are studied amongst others in \cite{spohn2004dynamics,bambusi1993some,kunze}.}, but also serves to engage the more pure mathematically inclined reader. The coupling of these two systems provides us with interesting non-linearities in the theory already when we are working with linear polysymplectic spaces. For notational simplicity, we consider only scalar fields and 2-dimensional periodic space-time for now. 

Let $q(t)$ denote the position of a particle moving in some external potential field $V_t$ on $S^1=\R/2\pi\Z$ and $\vphi(t,x)$ be a scalar field on $\T^2=(\R/T\Z)\times(\R/2\pi\Z)$ as before. We require $V_t=V_{t+T}$ to be smooth in $t$ and uniformly bounded by some constant $\mu$.  In order to avoid singularities, we assume the particle to be of finite size and let $\rho:\R\to\R$ model the coupling parameter between the particle and the field. We assume $\rho$ to be smooth and supported on some interval $(-r,r)$, where $r<\pi$ can be interpreted as the radius of the particle and $\rho$ describes the distribution of the charge over the particle.\footnote{Note that as $\rho$ is only supported on $(-r,r)\subseteq (-\pi,\pi]$, we might also view $\rho$ as a function on $S^1$ sending $x\in S^1=\R/2\pi\Z$ to $\rho(a)$, where $x = a \textup{ mod }2\pi$ and $a\in(-\pi,\pi]$. We will not make a notational difference between these two functions.} The equations of the coupled system look as follows (see \cite{particlefields}).
\begin{align}\label{eq:particlefield}\begin{split}
    \frac{d^2}{dt^2}q(t) &= -V'_t(q(t)) - \nabla(\vphi(t,\cdot)*\rho)(q(t))\\
    -\partial_t^2\vphi(t,x)+\partial_x^2\vphi(t,x)&=\vphi(t,x)+\rho(q(t)-x)
    \end{split}
\end{align}
Here $f*\rho(q):=\int_{S^1}f(x)\rho(q-x)\,dx$ for $f:S^1\to\R$. The coupled equations are described by two Hamiltonians: one for the mechanical system and one for the field. Let $u=(q,p)$, $Z=(\vphi,\pi_1,\pi_2,o)$ and define
\begin{align*}
    H_t(u,Z) &= \frac{1}{2}p^2+V_t(q)+(\vphi(t,\cdot)*\rho)(q)\\
    S_{t,x}(u,Z) &= \frac{1}{2}\vphi^2+\frac{1}{2}\pi_1^2-\frac{1}{2}\pi_2^2-\frac{1}{2}o^2+\vphi\cdot\rho(q(t)-x).
\end{align*}
Then \cref{eq:particlefield} can be reformulated as
\begin{align*}
    J\frac{d}{dt}u(t) &= \nabla_u H_t(u(t),Z(t,x))\\
    M_1\partial_tZ(t,x)+M_2\partial_xZ(t,x) &= \nabla_Z S_{t,x}(u(t),Z(t,x))
\end{align*}
where $\nabla_u$ and $\nabla_Z$ denote respectively the gradients with respect to the $u$ and $Z$ coordinates. We see that these equations indeed combine the symplectic Hamiltonian formalism and the Bridges regularized equations. Note that $u$ is now a map into $T^*S^1$ and, as before, $Z$ maps into $\R^4$.

To study these equations once again we can look at the Floer equation given by
\begin{align}\label{eq:particlefieldfloer}\begin{split}
     \partial_s \wt{u}(s,t) &= -J\partial_t\wt{u}(s,t)+\nabla_u H_t(\wt{u}(s,t),\wt{Z}(s,t,x))\\
     \partial_s \wt{Z}(s,t,x) &= - M_1\partial_t \wt{Z}(s,t,x) - M_2\partial_x \wt{Z}(s,t,x) + \nabla_Z S_{t,x}(\wt{u}(s,t),\wt{Z}(s,t,x)).
     \end{split}
\end{align}

\subsection{Asymptotics of Floer curves}\label{sub:asymptotics}
For the Floer curves of the combined system introduced above, we want to formulate and prove an analogue to lemma \ref{GW:ndg}. The energy of a Floer curve $\wt{F}:=(\wt{u},\wt{Z})$ is given as the sum of the energies of the components:
\[
E(\wt{F}) = \int_{-\infty}^{+\infty}\left(\int_{S^1}|\partial_s\wt{u}(s,t)|^2\,dt+\int_{\T^2}|\partial_s\wt{Z}(s,t,x)|^2\,dV\right)ds.
\]
As we are mostly interested in the field theory part of the equations in this article, we will assume the particle Floer curve $\wt{u}$ to be known. In a subsequent article the case of an unknown particle curve will be treated. 
\begin{theorem}\label{thm:asymp}
Let $\wt{F}=(\wt{u},\wt{Z})$ be a solution to \cref{eq:particlefieldfloer} of finite energy. We assume that there exists some smooth function $u^+\in C^\infty(S^1,T^*S^1)$, such that 
\begin{align*}
    \lim_{s\to+\infty}\wt{u}_s = u^+ && \lim_{s\to+\infty}\partial_s\wt{u}_s = 0
\end{align*}
in the $C^\infty$-topology. If $\ratio^{-2}=\left(\frac{T}{2\pi}\right)^2$ has irrationality measure\footnote{Compare with \cref{def:irrmeasure} and the assumption made in \cref{sub:degsolved}.} equal to 2, then there exists some $Z^+\in C^\infty(\T^2,\R^4)$ such that $\lim_{s\to+\infty}\wt{Z}_s=Z^+$ in the $C^\infty$-topology.
\end{theorem}
To prove \cref{thm:asymp}, we first introduce some notation. By $\hat{Z}_s$ and $\hat{\rho}^q_s$ we denote respectively the Fourier transforms of $\wt{Z}_s(t,x)$ and $\rho(\wt{q}_s(t)-x)$ with respect to both time and space. Since $\rho$ is smooth and $\wt{q}_s$ converges to $q^+$ in the $C^\infty$-topology, we get that $\lim_{s\to+\infty}\rho(\wt{q}_s(t)-x)$ is well defined. The limit of the Fourier series is denoted by $\hat{\rho}^q_+:=\lim_{s\to+\infty}\hat{\rho}^q_s$. Moreover, $\partial_s\rho(\wt{q}_s(t)-x)$ converges to 0 in $C^\infty$ for $s$ going to $+\infty$.

From the Floer equation we see that $\hat{Z}_s$ satisfies
\begin{align}\label{eq:FourierFloerPF}
    \partial_s\hat{Z}_s(m,k) = B(m,k)\hat{Z}_s(m,k)+\hat{\rho}^q_s(m,k)e_1,
\end{align}
for all $m,k\in\Z$. Here, $e_1=\begin{pmatrix}1&0&0&0\end{pmatrix}^T$ is the first basis vector of $\R^4$ and 
\begin{align*}
    B(m,k) = \begin{pmatrix}1&0&0&0\\0&1&0&0\\0&0&-1&0\\0&0&0&-1\end{pmatrix}-\ratio imM_1-ikM_2.
\end{align*}
The four distinct eigenvalues $\lambda_1^\pm(m,k)$ and $\lambda_2^\pm(m,k)$ of the Hermitian matrix $B(m,k)$ are given by 
\begin{align*}
    \lambda_1^+(m,k) &= -\lambda_1^-(m,k)=\ratio|m|+\sqrt{1+k^2}>1\\
    \lambda_2^+(m,k) &= -\lambda_2^-(m,k)=\left|\ratio|m|-\sqrt{1+k^2}\right|>0.
\end{align*}
\begin{proposition}\label{prop:eigenvalues}
There exists some constant $\wt{c}>0$ such that $\lambda_2^+(m,k)>\frac{\wt{c}}{(1+k^2)^2}$ for all $m,k\in\Z$. 
\end{proposition}
\begin{proof}
By assumption $\ratio^{-2}$ has irrationality measure 2, so for any $\delta>0$ there exists some $c>0$, such that 
\[
\left|\left(\frac{T}{2\pi}\right)^2-\frac{\alpha}{\beta}\right|>\frac{c}{\beta^{2+\delta}}
\]
for every $\frac{\alpha}{\beta}\in\Q$. In particlular we get for $\delta=\frac{1}{2}$ that there is some $c>0$ such that
\begin{align}\label{eq:mkineq}
\left|\frac{T}{2\pi}-\frac{|m|}{\sqrt{1+k^2}}\right|\left|\frac{T}{2\pi}+\frac{|m|}{\sqrt{1+k^2}}\right|>\frac{c}{(1+k^2)^{\frac{5}{2}}}
\end{align}
for all $m,k\in\Z$ (this follows from letting $\alpha=m^2$ and $\beta=1+k^2$). There are two cases to consider.\newline
\emph{Case 1:} If $\frac{|m|}{\sqrt{1+k^2}}\geq 1+\frac{T}{2\pi}$ then 
\begin{align*}
\lambda_2^+(m,k) &= \frac{2\pi}{T}|m|-\sqrt{1+k^2}\\
&\geq \frac{2\pi}{T}\sqrt{1+k^2}\\
&\geq\frac{2\pi}{T}\cdot\frac{1}{(1+k^2)^2}
\end{align*}
where the last step follows from the fact that $1+k^2\geq 1$. \newline
\emph{Case 2:} If $\frac{|m|}{\sqrt{1+k^2}}< 1+\frac{T}{2\pi}$ then
\[
\frac{T}{2\pi}+\frac{|m|}{\sqrt{1+k^2}}<1+\frac{T}{\pi}=:C.
\]
Filling this into \cref{eq:mkineq} gives
\[
\frac{c}{(1+k^2)^{\frac{5}{2}}}<C\left|\frac{T}{2\pi}-\frac{|m|}{\sqrt{1+k^2}}\right|.
\]
Thus
\begin{align*}
    \lambda_2^+(m,k) &= \left|\ratio|m|-\sqrt{1+k^2}\right|\\
    &>\frac{2\pi}{T}\cdot\frac{c}{C}\cdot\frac{1}{(1+k^2)^2}.
\end{align*}
\newline
In both cases we see that $\lambda_2^+(m,k)>\frac{\wt{c}}{(1+k^2)^2}$ when we pick $0<\wt{c}<\textup{max}\{\frac{2\pi}{T},\frac{2\pi c}{TC}\}$.
\end{proof}

\subsubsection*{Energy argument}
By assumption the energy of the Floer curve is finite. This implies the existence of a sequence of numbers $\{s_j\}$ diverging to $+\infty$, such that $\int_{\T^2}|\partial_s\wt{Z}(s_j,t,x)|^2\,dV$ converges to 0 for $j\to+\infty$. Note that $\int_{\T^2}|\partial_s\wt{Z}(s_j,t,x)|^2\,dV = ||\partial_s\hat{Z}_{s_j}||_{\ell^2}$.

\begin{lemma}\label{lem:dstozero}
We have $\lim_{s\to+\infty}\partial_s\wt{Z}_s=0$ in the $C^\infty$-topology.
\end{lemma}
\begin{proof}
We prove that the $L^1$-norm of all space-time derivatives of $\partial_s\wt{Z}_s$ converge to 0. Let $\epsilon>0$ and $p(m,k)$ be any polynomial with no integer roots. As we know that $\partial_s\rho(\wt{q}_s(t)-x)$ converges to 0 with all space-time derivative, there exists some $s^0$ such that for $s>s^0$ we have
\[
\sum_{m,k}|\partial_s\hat{\rho}^q_s(m,k)p(m,k)|^2<\epsilon^2.
\]
Then for all $m,k\in\Z$ it follows that 
\begin{align}\label{eq:derestim}
    |\partial_s\hat{\rho}^q_s(m,k)|<\frac{\epsilon}{|p(m,k)|}
\end{align} for $s>s^0$.

Define $W_s:=\partial_s\hat{Z}_s$. First, we want to prove that $|W_s(m,k)|<\frac{2\epsilon}{|p(m,k)|\lambda(m,k)}$ for all $s>s^0$ and $m,k\in\Z$, where $\lambda(m,k)$ is the smallest positive eigenvalue of $B(m,k)$.\footnote{The idea for this part of the proof is based on \cite[Lemma 7.2]{fabert2020hamiltonian}.} For the sake of contradiction suppose that there exists some $s^1>s^0$ such that $|W_{s^1}(m,k)|\geq \frac{2\epsilon}{|p(m,k)|\lambda(m,k)}$ for some $m,k\in\Z$. Then, if we write out $W_{s^1}(m,k)$ in an eigenbasis of $B(m,k)$, there is some component, denoted $W^i_{s^1}(m,k)$, belonging to eigenvalue $\lambda^i(m,k)$ for which
\[
|W_{s^1}^i(m,k)|\geq \frac{\epsilon}{|p(m,k)|\lambda(m,k)}.
\]
Without loss of generality, we may assume $W_{s^1}^i(m,k)>0$. Then by taking derivatives with respect to $s$ of \cref{eq:FourierFloerPF} we see that 
\[
\partial_s W^i_{s^1}(m,k) = \lambda^i(m,k)W^i_{s^1}(m,k)+(\partial_s\hat{\rho}^q_{s^1}(m,k)e_1)^i>0.
\]
As $W_{s^1}^i(m,k)>0$ and $W_{s_j}(m,k)$ converges to 0, there has to be a number $s^2>s^1$ with $\partial_s W^i_{s^2}=0$ and $\partial_s W^i_s(m,k)>0$ for all $s\in(s^1,s^2)$. So 
\begin{align*}
    W^i_{s^2}(m,k)>W^i_{s^1}(m,k)>\frac{\epsilon}{|p(m,k)|\lambda(m,k)}
\end{align*}
and
\begin{align*}
    0=\partial_s W^i_{s^2}(m,k)=\lambda^i(m,k)W^i_{s^2}(m,k)+(\partial_s\hat{\rho}^q_{s^2}(m,k)e_1)^i,
\end{align*}
contradicting \cref{eq:derestim}. This shows that indeed $|W_s(m,k)|<\frac{2\epsilon}{|p(m,k)|\lambda(m,k)}$ for all $s>s^0$ and $m,k\in\Z$.

When $\lambda(m,k)=\lambda_1^+(m,k)$ then $\lambda(m,k)^{-1}<1$ and when $\lambda(m,k)=\lambda_2^+(m,k)$ then $\lambda(m,k)^{-1}<\wt{c}^{-1}(1+k^2)^2$ by \cref{prop:eigenvalues}. So either way we see that $\lambda(m,k)^{-1}$ is at most polynomial in $m,k$. For all $l_1,l_2\in\Z$ we get that
\begin{align*}
    ||\partial_t^{l_1}\partial_x^{l_2}W_s(m,k)||_{L^1} &= \sum_{m,k}|W_s(m,k)k^{l_1}m^{l_2}|\\
    &<2\epsilon\sum_{m,k}\frac{\lambda(m,k)^{-1}k^{l_1}m^{l_2}}{|p(m,k)|}\\
    &<\epsilon\cdot\textup{const},
\end{align*}
where the last line follows from the fact that for fixed $l_1,l_2$ we can choose $p(m,k)$ such that $\frac{\lambda(m,k)^{-1}k^{l_1}m^{l_2}}{|p(m,k)|}<\frac{1}{m^2k^2}$.
This proves the lemma.
\end{proof}

\subsubsection*{The limit}
Now define $\hat{Z}^+(m,k):=-B(m,k)^{-1}\hat{\rho}^q_+(m,k)e_1$ and let $Z^+$ be the inverse Fourier transform of $\hat{Z}^+$. Note that the formula for $\hat{Z}^+$ makes sense, as $B(m,k)$ never has vanishing eigenvalues. First of all we would like to see that $Z^+$ is indeed a smooth function. 
\begin{proposition}
The function $Z^+$ defined above is a smooth function on $\T^2$.
\end{proposition}
\begin{proof}
We prove that all derivatives of $Z^+$ have finite $L^1$-norm. This implies that they are all continuous and thus that $Z^+$ is smooth. Let $l_1,l_2$ be integers. Then
\begin{align*}
    ||\partial_t^{l_1}\partial_x^{l_2}Z^+||_{L^1}=\sum_{m,k}|\hat{Z}^+(m,k)m^{l_1}k^{l_2}| &\leq \sum_{m,k}\lambda(m,k)^{-1}|m^{l_1}k^{l_2}\hat{\rho}^q_+(m,k)|,
\end{align*}
where $\lambda(m,k)$ is again the smallest positive eigenvalue of $B(m,k)$. Just like before, $\lambda(m,k)^{-1}$ is at most a polynomial expression in $m$ and $k$ and it follows from smoothness of $\rho(q^+(t)-x)$ that $||\partial_t^{l_1}\partial_x^{l_2}Z^+||_{L^1}<+\infty$.
\end{proof}

Finally we can finish the proof of \cref{thm:asymp} with the following lemma.
\begin{lemma}
For $Z^+$ defined as above it holds that $\lim_{s\to+\infty}\wt{Z}_s=Z^+$ in the $C^\infty$-topology.
\end{lemma}
\begin{proof}
Again, we will prove convergence by proving that the $L^1$-norms of all derivatives converge. Let $V_s(m,k)=\hat{Z}_s(m,k)-\hat{Z}^+(m,k)$. Then 
\begin{align}\label{eq:Vs}
    \partial_sV_s(m,k)=B(m,k)V_s(m,k)+(\hat{\rho}^q_s(m,k)-\hat{\rho}^q_+(m,k))e_1.
\end{align}
Let $\epsilon>0$ and $p(m,k)$ a polynomial with no integer roots. By \cref{lem:dstozero} and the fact that $\hat{\rho}^q_s$ converges to $\hat{\rho}^q_+$, there exists some $s^0$ such that for $s>s^0$
\begin{align*}
    \sum_{m,k} |\partial_sV_s(m,k)p(m,k)|^2 &< \epsilon\\
    \sum_{m,k} |(\hat{\rho}^q_s(m,k)-\hat{\rho}^q_+(m,k))p(m,k)|^2 &< \epsilon.
\end{align*}
For integers $l_1,l_2$, we have that
\begin{align}\label{eq:ZsminZplus}
   ||\partial_t^{l_1}\partial_x^{l_2}(\wt{Z}_s-Z^+)||_{L^1} &= \sum_{m,k}|V_s(m,k)m^{l_1}k^{l_2}| \\
   &< \sum_{m,k}|p(m,k)B(m,k)V_s(m,k)|\cdot|m^{l_1}k^{l_2}\frac{\lambda(m,k)^{-1}}{p(m,k)}|.
\end{align}
\Cref{eq:Vs} implies that
\[
|p(m,k)B(m,k)V_s(m,k)| \leq |\partial_sV_s(m,k)p(m,k)|+|(\hat{\rho}^q_s(m,k)-\hat{\rho}^q_+(m,k))p(m,k)|,
\]
which combined with \cref{eq:ZsminZplus} and the Cauchy-Schwarz inequality gives
\begin{align*}
    ||\partial_t^{l_1}\partial_x^{l_2}(\wt{Z}_s-Z^+)||_{L^1} &< \sum_{m,k} |\partial_sV_s(m,k)p(m,k)|^2\sum_{m,k}\left|m^{l_1}k^{l_2}\frac{\lambda(m,k)^{-1}}{p(m,k)}\right|^2\\
    &\,+\sum_{m,k} |(\hat{\rho}^q_s(m,k)-\hat{\rho}^q_+(m,k))p(m,k)|^2\sum_{m,k}\left|m^{l_1}k^{l_2}\frac{\lambda(m,k)^{-1}}{p(m,k)}\right|^2\\
    &< \epsilon\sum_{m,k}\left|m^{l_1}k^{l_2}\frac{\lambda(m,k)^{-1}}{p(m,k)}\right|^2
\end{align*}
for $s>s^0$.
Again we see by \cref{prop:eigenvalues} that we can pick $p(m,k)$ such that the last line is smaller than some constant multiple of $\epsilon$. So indeed the $L^1$ norms of all derivatives of $\wt{Z}_s-Z^+$ converge to 0, proving the lemma.
\end{proof}

\subsection{Lorentz model}\label{sub:Lorentz}
Notice that in \cref{eq:particlefield} it seems like we broke the symmetry of space and time. When introducing a particle this is inevitable since, from the viewpoint of the particle, space and time are split and the particle has a finite size only in space. However, if the particle would be moving at some different speed it might experience a different splitting of space and time. The particle equation would become
\begin{align*}
      \frac{d^2}{dt'^2}q(t') &= -V'_{t'}(q(t')) - \nabla(\vphi(t',\cdot)*\rho)(q(t')),
\end{align*}
where $t'$ is the new time coordinate. How would the field equation change? According to the Lorentz model\footnote{Note that we are dealing with a simplification of the Lorentz model here in which the speed of the particle is significantly bigger than its acceleration. This allows us to regard the reference frame of the particle to be constant. For an explanation of the Lorentz model in full generality, see \cite{spohn2004dynamics}}, explained in \cite{spohn2004dynamics}, the particle has a finite size in its own reference frame and thus the field equation becomes
\begin{align*}
      -\partial_t^2\vphi(t,x)+\partial_x^2\vphi(t,x)&=\vphi(t,x)+\rho(q(t')-x'),
\end{align*}
where $t'$ and $x'$ are the space-time coordinates in the rest frame of the particle. Note that $t'$ and $x'$ are functions of $t$ and $x$ and all the proofs of \cref{sub:asymptotics} still hold. The only thing about the function $\rho(q(t)-x)$ that we used in that section was that it was smooth and changing the space-time coordinates to $(t',x')$ does not change that. In fact, we could even have treated two or more particles moving in different reference frames and all interacting with the field. For two particles the equations look as follows
\begin{align*}
    \frac{d^2}{dt^2}q_1(t) &= -V'_{1,t}(q_1(t)) - \nabla(\vphi(t,\cdot)*\rho)(q_1(t))\\
    \frac{d^2}{dt'^2}q_2(t') &= -V'_{2,t'}(q_2(t')) - \nabla(\vphi(t',\cdot)*\rho)(q_2(t'))\\
    -\partial_t^2\vphi(t,x)+\partial_x^2\vphi(t,x)&=\vphi(t,x)+\rho(q_1(t)-x)+\rho(q_2(t')-x').
\end{align*}
Still, all the proofs of \cref{sub:asymptotics} hold. 

This is where we see a clear advantage of our covariant theory over \cite{paperoliverniek,particlefields}. Fabert and Lamoree need a fixed splitting of time and space to exist before starting with the analysis of the equations. They demand the non-linearities in the equations to satisfy conditions that explicitly depend on the chosen space-time coordinates. We see that in our theory, no such splitting a priori is required and we can deal with non-linearities coming from different choices of coordinates simultaneously.

\begin{remark}
In this section we fixed our attention to particles moving in 1-dimensional periodic space. However, the proofs generalize to higher dimensional spaces. In fact, for general $n$, we can take $q$ to be a particle constrained to any submanifold $Q$ of $\T^{n-1}$, where $\T^{n-1}$ denotes $(n-1)$-dimensional periodic space and all space periods are taken to be equal to $2\pi$. The field is defined on $\T^n=(\R/T\Z)\times \T^{n-1}$. Now, we can still prove that the eigenvalues of the linearized operator converge to 0 at most with polynomial speed, so that the same results hold as above. The proofs for general $n$ will be part of our subsequent article. This remark serves mostly to show that by looking at particles constrained to any $Q\subseteq \T^{n-1}$, we can couple our theory of CFT to any Hamiltonian system on a cotangent bundle. 
\end{remark}

\subsection{Existence of Floer curves}\label{sub:existence}
We finish this article with the proof that solutions to the particle-field Floer \cref{eq:particlefieldfloer} actually exist. This can be done by using results from Fabert and Lamoree \cite{paperoliverniek,particlefields,fredholm}. Recall from \cref{sub:infvp} that we can translate the Floer equation into the infinite-dimensional setting used by Fabert and Lamoree. Define $\Hil:=L^2_0(S^1,\R^4)$, the space of square-integrable functions with mean zero. After endowing $T^*S^1\times \Hil$ with the complex structure $\textup{Diag}(J,M_1)$, solutions to \cref{eq:particlefieldfloer} are in one-to-one correspondence with infinite-dimensional Floer curves coming from the Hamiltonian
\begin{align*}
    \niek{H}_t : T^*S^1\times \Hil &\to \R\\
    (u,Z)&\mapsto \frac{1}{2}\langle AZ,Z\rangle_{\Hil} +F_t(u,Z),
\end{align*}
where 
\begin{align}\label{eq:A}
    A = \begin{pmatrix}
    1&0&0&0\\0&1&0&0\\0&0&-1&0\\0&0&0&-1
    \end{pmatrix}-M_2\partial_x
\end{align}
and $F_t(u,Z)=\frac{1}{2}p^2+V_t(q)+(\vphi*\rho)(q)$. In order to use \cite[Theorem 10.4]{paperoliverniek} and conclude that a Floer curve exists, we must check that $(A,T)$ and $F_t$ are admissible according to definitions 2.1 and 2.5 in \cite{fredholm}. First, notice that since $M_2\partial_x$ and hence $A$ is self-adjoint, there exists a complete basis\footnote{This basis can be found by looking at the Fourier transform of $\Hil$.} $\{e_k^\pm\mid k\in \Z\backslash\{0\}\}$ of $\Hil$ consisting of unit eigenvectors of $A$ such that $M_1e_k^\pm=\mp e_k^\mp$ and the eigenvalues are given by $\lambda_{2k}=\sqrt{1+k^2}$ and $\lambda_{2k+1}=-\sqrt{1+k^2}$. This allows us to identify $\Hil$ with a complex subspace of $\Hil\otimes_\R \C$ given by the complex span of $\{z_k=(e^+_k+ie^-_k)/\sqrt{2} \mid k\in \Z\backslash\{0\}\}$. For any $h$ the space $\Hil_h$ can be defined by requiring that $||z_k||_{\Hil_h}=k^h$. First of all we check that $(A,T)$ is admissible. 
\begin{definition}[Compare definition 2.1 from \cite{fredholm}]
The pair $(A,T)$ is \emph{admissible} if there exists $h_0$ such that for $h>h_0$ we can find $c_h$ such that $|\epsilon_k|>c_hk^{-h}$. Here, $\epsilon_k\in(-\pi/T,\pi/T]$ is defined by $\epsilon_k:=\lambda_k \mod 2\pi/T$.
\end{definition}
\begin{lemma}
The pair $(A,T)$ with $A$ as in \cref{eq:A} and $T$ such that $\left(\frac{T}{2\pi}\right)^2$ has irrationality measure 2 is admissible. 
\end{lemma}
\begin{proof}
Note that $|\epsilon_{2k}|=\min\{\lambda_2^+(m,k)\mid m\in \Z\}$, where $\lambda_2^+(m,k)=|\frac{2\pi}{T}|m|-\sqrt{1+k^2}|$ as in \cref{sub:asymptotics}. As was proven in \cref{prop:eigenvalues}, there is some $\wt{c}$ such that $\lambda^+_2(m,k)>\frac{\wt{c}}{(1+k^2)^2}$ for all $m,k\in \Z$. Left to prove is that there exists $c_h$ such that $\frac{\wt{c}}{(1+k^2)^2}\geq c_h(2k)^{-h}$. This can always be done for $h>h_0:=3$. For $\epsilon_{2k+1}$ the proof is similar. 
\end{proof}
Secondly, we must check the admissibility of $F_t$.
\begin{definition}[Simplified definition 2.5 of \cite{fredholm}]\label{def:reg}
For a finite-dimensional symplectic manifold $M$, the $T$-periodic non-linearity $F_t:M\times\Hil\to\R$ is called \emph{$\infty$-regularizing} if for any $h>h_0$, the map $F_t$ extends to a smooth map $M\times \Hil_{-h}\to\R$ which is smooth in the $t$-variable. Any $\infty$-regularizing non-linearity is called \emph{admissible}. 
\end{definition}
\begin{lemma}
The nonlinearity $F_t$ defined above is $\infty$-regularizing.
\end{lemma}
\begin{proof}
Assume that $Z\in\Hil_{-h}$ for some $h>h_0$ and write $Z=\sum\hat{Z}(k)z_k$. Now by definition of $\Hil_{-h}$ it follows that $(\hat{Z}(k)k^{-h})\in\ell^2(\Z,\C^2)$. Also, for $q\in S^1$, define $\underline{\rho^q}=\rho(q-\cdot)\begin{pmatrix}1&0&0&0\end{pmatrix}^T=\sum\hat{\underline{\rho^q}}(k)z_k$. Note that since $\rho(q-\cdot)$ is smooth, $||\hat{\underline{\rho^q}}(k)k^h||_{\ell^2}<+\infty$ for all $h>0$. Then
\begin{align*}
    |(\vphi*\rho)(q)|&=\left|\int_{S^1}\vphi(x)\rho(q-x)\,dx\right|\\
    &=\left|\int_{S^1}\langle Z(x),\underline{\rho^q}(x)\rangle\,dx\right|\\
    &=\left|\langle Z,\underline{\rho^q}\rangle_\Hil\right|\\
    &=\left|\langle \hat{Z}(k),\hat{\underline{\rho^q}}(k)\rangle_{\ell^2}\right|\\
    &=\left|\langle \hat{Z}(k)k^{-h},\hat{\underline{\rho^q}}(k)k^h\rangle_{\ell^2}\right|\\
    &\leq ||\hat{Z}(k)k^{-h}||_{\ell^2} ||\hat{\underline{\rho^q}}(k)k^h||_{\ell^2} <+\infty.
%
\end{align*}
Here, the last line follows from Cauchy-Schwarz. In particular, it follows from this computation that $(\vphi*\rho)(q)\in\R$ and thus $F_t$ indeed extends to $T^*S^1\times \Hil_{-h}$. It is easy to see that the resulting map is smooth and also smooth in the $t$-variable (compare with proposition 2.8 from \cite{fredholm}).
\end{proof}

The last thing to check is boundedness in the $p$-coordinate. As we are working with $T^*S^1$ for our particle which is not compact, it must be checked that this does not cause any divergence. Since our Hamiltonian is quadratic in $p$, this does not pose any problems. Conditions (F1) and (F2) from section 2 of \cite{particlefields} can be easily checked and we refer to the explanation in that article for the fact that this suffices. 

Now by combining the admissibility of $(A,T)$ and $F_t$ and the boundedness in $p$ with the fact that $T^*S^1$ has non-trivial homology, theorem 10.4 from \cite{paperoliverniek} tells us that Floer curves exist for \cref{eq:particlefieldfloer}. Even though the field equations are linear, this result is still non-trivial because of the interaction with the particle non-linearity. Note that the translation to the infinite-dimensional setting of Fabert and Lamoree is possible for \cref{eq:particlefieldfloer}, because the non-linearity has a specific form. In a subsequent article we will investigate the existence of of Floer curves in our covariant model for a larger class of polysymplectic Hamiltonians. 

\subsection*{Concluding remarks}
It should be clear from the discussion above that the analysis done in \cite{paperoliverniek,fredholm,particlefields} on the infinite-dimensional framework significantly supports the development of our covariant framework. The goal is to translate the Fredholm theory and compactness results from the aforementioned articles into the covariant setting and thus developing a covariant theory that stands on its own. These two main ingredients, combined with the work done in this article, will allow us to define a Floer theory that incorporates the symmetries in time and space, and is capable of dealing with relativistic models, such as the Lorentz model discussed above.

\bibliography{mybib}{}

\begin{thebibliography}{RRMRSV11}

\bibitem[AD14]{audindamian}
Michele Audin and Mihai Damian.
\newblock {\em Morse theory and Floer homology}.
\newblock Springer, 2014.

\bibitem[BG93]{bambusi1993some}
Dario Bambusi and Luigi Galgani.
\newblock Some rigorous results on the pauli-fierz model of classical
  electrodynamics.
\newblock In {\em Annales de l'IHP Physique th{\'e}orique}, volume~58, pages
  155--171, 1993.

\bibitem[Bri06]{bridgesTEA}
Thomas~J Bridges.
\newblock Canonical multi-symplectic structure on the total exterior algebra
  bundle.
\newblock {\em Proceedings of the Royal Society A: Mathematical, Physical and
  Engineering Sciences}, 462(2069):1531--1551, 2006.

\bibitem[Bug12]{diophantine}
Yann Bugeaud.
\newblock {\em Distribution modulo one and Diophantine approximation}, volume
  193.
\newblock Cambridge University Press, 2012.

\bibitem[DS08]{da2008lectures}
Ana~Cannas Da~Silva.
\newblock {\em Lectures on symplectic geometry}.
\newblock Springer, 2008.

\bibitem[Fab20]{fabert2020hamiltonian}
Oliver Fabert.
\newblock Hamiltonian floer theory for nonlinear schr\"odinger equations and
  the small divisor problem, 2020.

\bibitem[FG13]{multipoly_forgergomes}
Michael Forger and Leandro~G Gomes.
\newblock Multisymplectic and polysymplectic structures on fiber bundles.
\newblock {\em Reviews in Mathematical Physics}, 25(09):1350018, 2013.

\bibitem[FL21a]{particlefields}
Oliver Fabert and Niek Lamoree.
\newblock Cuplength estimates for periodic solutions of hamiltonian
  particle-field systems.
\newblock {\em arXiv preprint 2107.03989}, 2021.

\bibitem[FL21b]{fredholm}
Oliver Fabert and Niek Lamoree.
\newblock Floer homology for hamiltonian pdes: Fredholm theory.
\newblock {\em arXiv preprint 2107.14074}, 2021.

\bibitem[FL22]{paperoliverniek}
Oliver Fabert and Niek Lamoree.
\newblock Time-periodic solutions of hamiltonian pdes using pseudoholomorphic
  curves.
\newblock {\em Algebraic and Geometric Topology}, 2022.

\bibitem[Flo88]{floer1988morse}
Andreas Floer.
\newblock Morse theory for lagrangian intersections.
\newblock {\em Journal of differential geometry}, 28(3):513--547, 1988.

\bibitem[G{\"u}n87]{gunther1987polysymplectic}
Christian G{\"u}nther.
\newblock The polysymplectic hamiltonian formalism in field theory and calculus
  of variations. i. the local case.
\newblock {\em Journal of differential geometry}, 25(1):23--53, 1987.

\bibitem[H{\'e}l01]{helein2001hamiltonian}
F~H{\'e}lein.
\newblock Hamiltonian formalisms for multidimensional calculus of variations
  and perturbation theory.
\newblock {\em arXiv preprint math-ph/0212036}, 2001.

\bibitem[Kan93]{kanatchikov1993canonical}
Igor~V Kanatchikov.
\newblock On the canonical structure of the de donder-weyl covariant
  hamiltonian formulation of field theory i. graded poisson brackets and
  equations of motion.
\newblock {\em arXiv preprint hep-th/9312162}, 1993.

\bibitem[Kru02]{KRUPKOVA200293}
Olga Krupková.
\newblock Hamiltonian field theory.
\newblock {\em Journal of Geometry and Physics}, 43(2):93--132, 2002.

\bibitem[Kun01]{kunze}
M.~Kunze.
\newblock On the period of periodic motions of a particle in a scalar wave
  field.
\newblock {\em ZAMM Journal of applied mathematics and mechanics: Zeitschrift
  f\"ur angewandte Mathematik und Mechanik}, 81, 01 2001.

\bibitem[McC21]{mcclain2021global}
Tom McClain.
\newblock A global version of g{\"u}nther’s polysymplectic formalism using
  vertical projections.
\newblock {\em Journal of Geometry and Physics}, 161:104065, 2021.

\bibitem[MPS98]{marsden1998multisymplectic}
Jerrold~E Marsden, George~W Patrick, and Steve Shkoller.
\newblock Multisymplectic geometry, variational integrators, and nonlinear
  pdes.
\newblock {\em Communications in Mathematical Physics}, 199(2):351--395, 1998.

\bibitem[MS12]{mcduff2012j}
Dusa McDuff and Dietmar Salamon.
\newblock {\em J-holomorphic curves and symplectic topology}, volume~52.
\newblock American Mathematical Soc., 2012.

\bibitem[RRMRSV11]{royksymplectic}
Narciso Román-Roy, Ángel M.~Rey, Modesto Salgado, and Silvia Vilariño.
\newblock On the $k$-symplectic, $k$-cosymplectic and multisymplectic
  formalisms of classical field theories.
\newblock {\em Journal of Geometric Mechanics}, 3(1):113–137, 2011.

\bibitem[Sal99]{salamon1999lectures}
Dietmar Salamon.
\newblock Lectures on floer homology.
\newblock {\em Symplectic geometry and topology (Park City, UT, 1997)},
  7:143--229, 1999.

\bibitem[Spo04]{spohn2004dynamics}
Herbert Spohn.
\newblock {\em Dynamics of charged particles and their radiation field}.
\newblock Cambridge university press, 2004.

\end{thebibliography}
\bibliographystyle{alpha}

\end{document}